\documentclass[11pt]{amsart} 
\usepackage[lmargin=1in,rmargin=1in,tmargin=1in,bmargin=1in]{geometry}
\usepackage[ps,all,arc,rotate]{xy}
\usepackage{graphicx, float, epstopdf}
\usepackage{hyperref, color}
\usepackage{centernot}
\usepackage{fancyhdr}
\usepackage[utf8]{inputenc}
\usepackage{amsfonts,amssymb,amsmath,amsthm,mathrsfs}
\usepackage{graphics, setspace}
\usepackage{braket}
\usepackage{mathtools}
\usepackage{tikz}

\numberwithin{equation}{section}
\numberwithin{figure}{section}
\allowdisplaybreaks[4]         %allows breaking multiline environments in amsmath commands
 
\newtheorem{lemma}{Lemma}[section]
\newtheorem{theorem}{Theorem}[section]
\newtheorem{proposition}{Proposition}[section]

\newtheorem{corollary}[lemma]{Corollary}
\theoremstyle{definition}

\newtheorem{remark}{Remark}[section]
\newcommand{\Z}{\mathbb{Z}}

\newcommand{\R}{\mathbb{R}}
\newcommand{\C}{\mathbb{C}}
\newcommand{\N}{\mathbb{N}}
\newcommand{\Nset}{\mathcal{N}}
\newcommand{\Mset}{\mathcal{M}}

\newcommand{\Pri}{\mathbb{P}}
\newcommand{\partition}{\mathfrak{p}}

\newcommand{\card}{\operatorname{card}}

\newcommand{\e}{\operatorname{e}}

\newcommand{\real}{\operatorname{Re}}
\newcommand{\imag}{\operatorname{Im}}

\newcommand{\sumtwo}{\operatorname*{\sum\sum}}

\newcommand{\modu}{\operatorname{mod}}

\newcommand{\mycomment}[1]{}

\begin{document}
\title{Partitions into semiprimes} 
%%%%%%%%%%%%%%%
\author[M.~Das]{Madhuparna Das}
\address{Department of Mathematics, University of Exeter, Exeter, EX4 4QF, United Kingdom}
\email{md679@exeter.ac.uk}
\author[N.~Robles]{Nicolas Robles}
\address{
	\textnormal{Previously at} Department of Mathematics, University of Illinois, 1409 West Green Street, Urbana, IL 61801, USA; \textnormal{current address} 
	IBM Quantum, IBM T. J. Watson Research Center, Yorktown Heights, New York 10598, USA}
\email{nicolas.robles@ibm.com}
\author[A.~Zaharescu]{Alexandru Zaharescu}
\address{Department of Mathematics, University of Illinois, 1409 West Green
Street, Urbana, IL 61801, USA; \textnormal{and} Institute of Mathematics of the Romanian academy, P.O. BOX 1-764, Bucharest, Ro-70700, Romania}
\email{zaharesc@illinois.edu}
\author[D.~Zeindler]{Dirk Zeindler}
\address{Department of Mathematics and Statistics, Lancaster University, Fylde College, Bailrigg, Lancaster LA1 4YF, United Kingdom}
\email{d.zeindler@lancaster.ac.uk}
%%%%%%%%%%%%%%%
\subjclass[2020]{Primary: 11P55; 11L03, 11P82, 11L20. Secondary: 11M41. \\ \indent \textit{Keywords and phrases}: partitions, semiprimes, Hardy-Littlewood circle method, Weyl sums, Vinogradov's bound, prime zeta function, logarithmic singularities, Meissel-Mertens constant.}
\maketitle
%%%%%%%%%%%%%%%%%%%%%%%%%%%%%%%%%%%%%%%%%%%%%%%%%%%%%%%%%%%%%%%%%%%%%%%%%%%%%%%%%%%%%%%%%%%%%%%%%
\begin{abstract}
Let $\mathbb{P}$ denote the set of primes and $\Nset \subset \N$ be a set with arbitrary weights attached to its elements. Set $\partition_{\Nset}(n)$ to be the restricted partition function which counts partitions of $n$ with all its parts lying in $\Nset$. By employing a suitable variation of the Hardy-Littlewood circle method we provide the asymptotic formula of $\partition_{\Nset}(n)$ for the set of semiprimes $\Nset = \{p_1 p_2 : p_1, p_2 \in \mathbb{P}\}$ in different set-ups (counting factors, repeating the count of factors, and different factors). In order to deal with the minor arc, we investigate a double Weyl sum over prime products and find its corresponding bound thereby extending some of the results of Vinogradov on partitions. We also describe a methodology to find the asymptotic partition $\partition_{\Nset}(n)$ for general weighted sets $\Nset$ by assigning different strategies for the major, non-principal major, and minor arcs. Our result is contextualized alongside other recent results in partition asymptotics. 
\end{abstract}
%%%%%%%%%%%%%%%%%%%%%%%%%%%%%%%%%%%%%%%%%%%%%%%%%%%%%%%%%%%%%%%%%%%%%%%%%%%%%%%%%%%%%%%%%%%%%%%%%
%\tableofcontents
%%%%%%%%%%%%%%%%%%%%%%%%%%%%%%%%%%%%%%%%%%%%%%%%%%%%%%%%%%%%%%%%%%%%%%
\section{Introduction and results} \label{sec:introduction}
%%%%%%%%%%%%%%%%%%%%%%%%%%%%%%%%%%%%%%%%%%%%%%%%%%%%%%%%%%%%%%%%%%%%%%
A partition of a positive integer $n$ is a non-decreasing sequence of positive integers whose sum is $n$. Suppose that $\mathcal{N} \subseteq \mathbb{N}$ and let $\partition_\Nset(n)$ denote the restricted partition function which counts partitions of $n$ lying within $\Nset$. When $\Nset = \N$, we obtain the unrestricted partition function $\partition_{\mathbb{N}}(n)$ studied by Hardy and Ramanujan \cite{hardyramanujuan} in 1918 through the use of the Hardy-Littlewood method. Their result states that
\begin{align*}
\partition_{\mathbb{N}}(n) \sim \frac{1}{4n\sqrt{3}} \exp \bigg(\pi \sqrt{\frac{2n}{3}}\bigg) \quad \textnormal{as} \quad n \to \infty.
\end{align*}
Partitions into prime numbers have been studied by various authors such as Bateman and Erd\H{o}s \cite{batemanerdos1, batemanerdos2}, Browkin \cite{browkin}, Kerawala \cite{kerawala}, Roth and Szekeres \cite{rothszekeres}, as well as Yang \cite{yang}. In 2008 Vaughan \cite{vaughan} was able to simplify and improve most of the literature on prime partitions. Let $\mathbb{P} = \{p: p \textnormal{ is prime}\}$. We will now study the partitions into semiprimes. Our main result is as follows.
\begin{theorem} \label{maintheorempapernotrepeated}
Let $\mathbb{P}_2 = \{ p_1 p_2 : p_1,p_2 \in \mathbb{P}\}$ and semiprimes are counted only once, i.e. $p_1p_2$ and $p_2p_1$ are considered the same. The number $\partition_{\mathbb{P}_2}$ of such partitions of $n$ with all parts lying in $\mathbb{P}_2$ satisfies
\begin{align*}
\partition_{\mathbb{P}_2}(n)
&\sim \mathfrak{c}_1 n^{-\frac{3}{4}} (\log(n/2))^{-\frac{1}{4}} (\log \log n + \mathfrak{c}_2)^{\frac{1}{4}} \nonumber \\
&\quad \times \exp \bigg\{\mathfrak{c}_3 \bigg( \frac{n}{\log (n/2)}\bigg)^{\frac{1}{2}} (\log \log n + \mathfrak{c}_2)^{\frac{1}{2}} \bigg(1+ O\bigg(\frac{\log \log n}{\log n}\bigg)\bigg)\bigg\}, \nonumber
\end{align*}
as $n \to \infty$ and the constants are given by
\begin{align*}
\mathfrak{c}_1 = \frac{(4 \zeta(2))^{\frac{1}{4}}}{2^{\frac{1}{4}} \sqrt{\pi}}, \quad \mathfrak{c}_2 = M - \log 2, 
\quad \textnormal{and} \quad \mathfrak{c}_3 = (2^{\frac{1}{2}}+2^{-\frac{3}{2}})(4 \zeta(2))^{\frac{1}{2}}.
\end{align*}
Here $M$ is the Meissel-Mertens constant $M = \gamma + \sum_{p \in \mathbb{P}} (\log(1-\frac{1}{p})+\frac{1}{p}) \approx 0.26149721 \ldots$ where $\gamma$ is the Euler constant.
\end{theorem}
The technique we present is versatile enough to answer questions about other types of restricted sets of primes. 
\begin{theorem} \label{theoremO2}
Let $\mathbb{P}_2^{\neq} = \{ p_1 p_2 : p_1,p_2 \in \mathbb{P}, p_1\neq p_2\}$. The number $\partition_{\mathbb{P}_2^{\neq}}$ of such partitions of $n$ with all parts lying in $\mathbb{P}_2^{\ne}$ satisfies
\begin{align*}
\partition_{\mathbb{P}_2^{\neq}}(n) 
&\sim \mathfrak{c}_1 n^{-\frac{3}{4}} (\log(n/2))^{-\frac{1}{4}} (\log \log n + \mathfrak{c}_2)^{\frac{1}{4}} \nonumber \\
& \quad \times \exp \bigg\{\mathfrak{c}_3 \bigg( \frac{n}{\log(n/2)}\bigg)^{\frac{1}{2}} (\log \log n + \mathfrak{c}_2)^{\frac{1}{2}} \bigg(1+ O\bigg(\frac{\log \log n}{\log n}\bigg)\bigg)\bigg\}, \nonumber
\end{align*}
as $n \to \infty$ and where the constants are the same as in Theorem \textnormal{\ref{maintheorempapernotrepeated}}.
\end{theorem}
%
%\begin{color}{red}
%Now set $\mathbb{P}_2^\sharp = \{p_1 p_2 : (p_1,p_2) \textnormal{ and } p_1,p_2 \in \mathbb{P}\}$. 
Now set $\mathbb{P}_2^\sharp =\{ (p_1,p_2)\in\N^2 : p_1,p_2\in\mathbb{P}\}$. 
Further, we identify an element $(p_1,p_2)\in\mathbb{P}_2^\sharp$ with the part $p_1p_2$.
Here semiprimes are counted twice since $(p_1,p_2) \neq (p_2,p_1)$ unless $p_1=p_2$. 
The number of partitions with all parts in $\mathbb{P}_2^\sharp(n)$ is denoted by $\partition_{\mathbb{P}^2}^\sharp(n)$ and
the asymptotic behaviour of $\partition_{\mathbb{P}^2}^\sharp(n)$ is stated in Theorem~\ref{maintheorempaper} below.
The techniques employed to prove Theorem~\ref{maintheorempapernotrepeated}, Theorem~\ref{theoremO2}, and Theorem~\ref{maintheorempaper} are the same.
In fact, we can prove all three theorems in one single bundle by studying a slightly more general case covering all three results simultaneously. We do this by assigning to squares and non-squares different weights in the corresponding generating function, see Section~\ref{sec:Gen_function_used}.
%\end{color} 

%Now set $\mathbb{P}_2^\sharp(n) = \{ p_1 p_2 : p_1,p_2 \in \mathbb{P}, p_1p_2 = p_2p_1\}$ in which semiprimes are counted twice due to the ordering of the factors $p_1 p_2 = p_2 p_1$. The techniques employed to prove Theorem \ref{maintheorempapernotrepeated} can also be repurposed to estimate the growth of $\partition_{\mathbb{P}^2}^\sharp(n)$. Indeed, Theorem \ref{theoremO2} distinguishes between the order of the factors in the product of primes, e.g. the pair $2 \times 3$ and $3 \times 2$ is only counted once. In order to prove Theorem \ref{maintheorempapernotrepeated} we shall first need to count the semiprimes with repetitions and then remove squares of single primes. In order to do so, we shall need the following result.
%
\begin{theorem} \label{maintheorempaper}
Let $\mathbb{P}_2^\sharp(n)$ be defined as above. The number $\partition_{\mathbb{P}_2^\sharp}$ of such partitions of $n$ with all parts lying in $\mathbb{P}_2^\sharp$ satisfies
\begin{align*}
\partition_{\mathbb{P}_2^\sharp}(n) &\sim \mathfrak{c}_1 n^{-\frac{3}{4}} (\log n)^{-\frac{1}{4}} (\log \log n + \mathfrak{c}_2 )^{\frac{1}{4}} \nonumber \\
&\quad \times \exp \bigg\{\mathfrak{c}_3 \bigg( \frac{n}{\log n}\bigg)^{\frac{1}{2}} (\log \log n + \mathfrak{c}_2)^{\frac{1}{2}} \bigg(1+O\bigg(\frac{\log \log n}{\log n}\bigg)\bigg)\bigg\},
\end{align*}
as $n \to \infty$  and where the constants are the same as in Theorem \textnormal{\ref{maintheorempapernotrepeated}}.
\end{theorem}
\begin{theorem} \label{differencetheorempaper}
Using the notation and definitions above, we have
\begin{align*}
\mathfrak{p}_{\mathbb{P}_2}(n+1) - \mathfrak{p}_{\mathbb{P}_2}(n) \sim \mathfrak{c}_4\bigg(\frac{\mathfrak{c}_2 + \log \log n + o(1)}{n \log (n/2)}\bigg)^{\frac{1}{2}}\mathfrak{p}_{\mathbb{P}_2}(n)
\end{align*}
as $n \to \infty$ and where the constants are the same as in Theorem \textnormal{\ref{maintheorempapernotrepeated}} with $\mathfrak{c}_4 = 2\pi\sqrt{\frac{1}{3}}$.
\end{theorem}
\subsection{Previous results in the literature}
Over the last decade there has been a strong interest in studying partitions by employing the Hardy-Littlewood circle method and the technique pioneered in \cite{vaughan} and then in \cite{vaughansquares}. 
An important result using this technique was established by Gafni \cite{gafniintegerpowers} in 2016 and generalized to arithmetic progressions in 2018 by Berndt, Malik, and Zaharescu \cite{berndtmalikzaharescu}.
\begin{theorem}[Gafni]
If $\Nset_k = \{x^k: x \in \N, k \in \N_{\ge 2}\}$, then as $n \to \infty$ one has that
\begin{align*}
\partition_{\Nset_k}(n) \sim \mathfrak{c}_1 \exp(\mathfrak{c}_2 n^{\frac{1}{k+1}})n^{-\frac{3k+1}{2(k+1)}},
\end{align*}
where $\mathfrak{c}_1$ and $\mathfrak{c}_2$ are positive constants that can be made explicit and depend only on $k$.
\end{theorem}
\begin{theorem}[Berndt-Malik-Zaharescu]
If $\Nset_{k,a,b} = \{x^k: x \equiv a \modu b, k \in \N_{\ge 2}, (a,b)=1\}$, then as $n \to \infty$ one has that
\begin{align*}
\partition_{\Nset_{k,a,b}}(n) \sim \mathfrak{c}_1 \exp(\mathfrak{c}_2 n^{\frac{1}{k+1}})n^{-\frac{b+bk+2ak}{2b(k+1)}},
\end{align*}
where $\mathfrak{c}_1$ and $\mathfrak{c}_2$ are positive constants that can be made explicit and depend only on $a,b$ and $k$.
\end{theorem}
One can think of the sets $\Nset_{k}$ and $\Nset_{k,a,b}$ as integer values of the $x^k$ and $(a+bx)^k$ with $(a,b)=1$, respectively. Therefore this begged the question of whether these results could be extended to polynomials of arbitrary degree. This was accomplished in \cite{dr} by considering an exotic zeta function.
\begin{theorem}[Dunn-Robles]
Let $f$ be the polynomial of degree $d$ with $d \ge 2$ and set $\Nset_f = \{f(x): x \in \N \}$. If $\Nset_f \subset \N$ and $\gcd(\Nset_f)=1$, then as $n \to \infty$ one has that
\begin{align*}
\partition_{\Nset_f}(n) \sim \mathfrak{c}_1 \exp(\mathfrak{c}_2 n^{\frac{1}{d+1}})n^{-\frac{2d(1-\zeta(0,\alpha))+1}{2(d+1)}},
\end{align*}
where $\zeta(0,\alpha)$ is a value of an appropriate Matsumoto-Weng zeta function \cite{matsumotoweng}. The positive constants $\mathfrak{c}_1$ and $\mathfrak{c}_2$ can be made explicit and depend only on $f$.
\end{theorem}
Other very recent results include \cite{berndtrobleszaharescuzeindler} and \cite{bbbf}. We now move on to restricted partitions over primes and we state an important result of Gafni \cite{gafniprimepowers} from 2021.
\begin{theorem}[Gafni]
If $\mathbb{P}^k = \{p^k:p \in \mathbb{P}, k \in \mathbb{N}\}$, then as $n \to \infty$ one has that
\begin{align*}
\partition_{\mathbb{P}^k}(n) \sim \mathfrak{c}_1 \exp\bigg(\mathfrak{c}_2 \frac{n^{\frac{1}{k+1}}}{(\log n)^{\frac{k}{k+1}}}(1+o(1))\bigg)n^{-\frac{2k+1}{2k+2}} (\log n)^{-\frac{k}{2k+2}},
\end{align*}
where $\mathfrak{c}_1$ and $\mathfrak{c}_2$ are positive constants that can be made explicit and depend only on $k$.
\end{theorem}
The case $k=1$ had been settled by Vaughan \cite{vaughan} in 2008. The case $k=2$ will be needed to remove the squared primes from our count of duplicated semiprimes.

A very interesting recent development is due to Debryune and Tenenbaum \cite{debruynetenenbaum} by using the saddle-point method for restricted partitions. One can think of the saddle-point method as a coarse version of the circle method, see \cite{gafniprimepowers}, in which only the major arc is the principal arc at the origin. Debryune and Tenenbaum were able to show asympototic formulas for monomials $\partition_{\Nset_k}(n)$, monomials in arithmetic progressions $\partition_{\Nset_{k,a,b}}(n)$, and polynomials $\partition_{\Nset_f}(n)$, thereby encompassing some of the results above. However, as explained by Gafni in \cite{gafniprimepowers} sets like the set of primes, prime powers (or semiprimes) cannot be studied with the technique provided in \cite{debruynetenenbaum}. The reason is that one of the requirements is that the resulting Dirichlet series for the major arcs (see Section \ref{sec:strategy} below) must be meromorphically continued to the half-plane $\real(s) \ge -\varepsilon$ for $\varepsilon>0$. In our case of semiprimes the resulting Dirichlet series is $(\zeta_{\mathcal{P}}(s))^2$ where 
\begin{align} \label{primezetamoebius}
\zeta_{\mathcal{P}}(s) = \sum_{p \in \mathbb{P}} \frac{1}{p^s} = \sum_{n=1}^\infty \frac{\mu(n)}{n} \log \zeta(ns),
\end{align} 
which has essential singularities at every zero of the Riemann zeta-function. In addition, the line $\real(s)=0$ represents a natural boundary as the singularities cluster near all points of this line. Indeed, in \cite{landauwalfisz} Landau and Walfisz showed that $\zeta_{\mathcal{P}}(s)$ cannot be continued beyond the line $\real(s) = 0$ due to the fact that we face a clustering of singular points along the imaginary axis emanating from the non-trivial zeros of $\zeta(s)$ on the critical line. Therefore, although the method put forward in \cite{debruynetenenbaum} is very powerful, it does not cover certain difficult cases and in the next section we propose a method, partially based on \cite{gafniprimepowers}, to deal with sets that yield Dirichlet series with essential singularities and branch cuts. 

We shall also show that by employing tailored versions of the prime number theorem (in our case an asymptotic for $\sum_{p_1 p_2 \le x} 1$) along with Mertens' estimates also provides an alternative approach -- and sometimes quicker way -- of evaluating the principal major arcs. However, this methodology requires a minimum of information on the error of the standard prime number theorem. Using elementary methods to obtain the prime number theorem is sufficient to get the order of magnitude of the main asymptotic term but further refinements from the zero-free region lead to more accurate main and error terms. This provides an implicit way of linking information about the zeros of $\zeta(s)$ and partitions.

Throughout the paper we set the notation $\e(x) = \exp(2 \pi i x)$. We shall also use the convention that $\varepsilon$ denotes an arbitrarily small positive quantity that may not be the same at each occurrence. The logarithmic integral will be defined by the Cauchy principal value
\begin{align*}
\operatorname{li}(x) := \mathop {\lim }\limits_{\varepsilon \to 0+} \bigg(\int_0^{1-\varepsilon}+\int_{1+\varepsilon}^x\bigg) \frac{dt}{\log t} \quad \textnormal{and} \quad \operatorname{Li}(x) := \int_{2}^\infty \frac{dt}{\log t} = \operatorname{li}(x) - \operatorname{li}(2).
\end{align*}
Euler's constant will be denoted by $\gamma = C_0 \approx 0.577216 \ldots$; Fr\"{o}berg's constant will be denoted $\mathcal{D} = \sum_{n \ge 2} \frac{\mu(n)}{n} \log \zeta(n) \approx -0.315718 \ldots$. Finally, Mertens' first constant is $M_1 = M \approx 0.261497 \ldots$ and it comes from $\sum_{p \le x} \frac{1}{p} = \log \log x + M_1 + o(1)$ where the sum is taken over primes. \\

The authors are sincerely grateful to Taylor Daniels from Purdue University for producing the picture of the generating function in Figure \ref{fig:generatingfunction}.
%%%%%%%%%%%%%%%%%%%%%%%%%%%%%%%%%%%%%%%%%%%%%%%%%%%%%%%%%%%%%%%%%%%%%%
\section{Strategies for restricted weighted partitions} \label{sec:strategy}
%%%%%%%%%%%%%%%%%%%%%%%%%%%%%%%%%%%%%%%%%%%%%%%%%%%%%%%%%%%%%%%%%%%%%%
The generating function for partitions with weights into members of a general set $\Nset$ is given by
\begin{align*}
\Psi_{\{\Nset,w\}} (z) := \sum_{n=1}^\infty \partition_{\{\Nset,w\}}(n) z^n = \prod_{n \in \Nset} (1-z^n)^{-w(n)}.
\end{align*}
Here $w(n) \in \N$ is the weight placed on each element $n \in \mathcal{N}$. It is useful to deal with the logarithm of this function
\begin{align} \label{strategy1}
\Phi_{\{\Nset,w\}}(z) := \log \Psi_{\{\Nset,w\}}(z) = \sum_{j=1}^\infty \sum_{n \in \Nset} \frac{w(n)}{j} z^{nj}.
\end{align}
An application of Cauchy's integral formula yields
\begin{align} \label{strategy2}
\mathfrak{p}_{\{\Nset,w\}}(n) = \rho^{-n} \int_{0}^{1} \exp(\Phi_{\{\Nset,w\}}(\rho \e(\alpha))) \e(\alpha) d\alpha = \rho^{-n} \int_{0}^{1} \Psi_{\{\Nset,w\}}(\rho \e(\alpha)) \e(\alpha) d\alpha,
\end{align}
where $\rho < 1$. Therefore our problem at hand consists in evaluating $\Phi_{\{\Nset,w\}}(z)$ as given by the middle side of \eqref{strategy2}, insert into \eqref{strategy1} and compute the resulting integral. Usually when one implements the circle method, the integral in \eqref{strategy2} has to be divided into major and minor arcs. The major arcs will typically contribute to the main term of the asymptotic formula whereas one has to show that the minor arcs are error terms of smaller magnitude. However, when dealing with restricted weighted partitions, the contribution from the major arc at the origin is substantially larger than the contributions from the rest of the major arcs away from the origin \cite{gafniprimepowers, vaughan}. This implies that we split our integral as follows
\begin{align*}
\mathfrak{p}_{\{\Nset,w\}}(n) = \rho^{-n}\bigg(\int_{\mathfrak{M}(1,0)}+\int_{\mathfrak{M} \backslash \mathfrak{M}(1,0)} + \int_{\mathfrak{m}}\bigg) \exp(\Phi_{\{\Nset,w\}}(\rho \e(\alpha))) \e(\alpha) d\alpha.
\end{align*}
The first integral above, representing the major arcs, is treated by the use of contour integration or a by variation of the prime number theorem then followed by the saddle-point method. The major arcs for $\mathfrak{M}(q,a)$ with $q>1$ can be shown to yield a term smaller than $\mathfrak{M}(1,0)$. This means that the main term of the asymptotic $\mathfrak{p}_{\{\Nset,w\}}(n)$ will be dictated only by the first integral when $\alpha$ is close to the origin.

In order to evaluate the main arc at the origin, one needs detailed information about the set $\Nset$ and its corresponding weights $w(n)$. The technique consists in writing the associated Dirichlet function of the set $\Nset$ as a function of known zeta functions (e.g. Riemann, Hurwitz, Dirichlet, Matsumoto-Weng, or prime zeta functions). In other words, we will need to write
\begin{align*}
\zeta_{\{\Nset, w\}}(s) := \sum_{n \in \Nset} \frac{w(n)}{n^s} = f(\zeta(s), \zeta(s,a), \zeta_{\mathcal{P}}(s), L(s,\chi), \cdots),
\end{align*}
for some function $f$ to be determined. For instance, most unrestricted partitions can be mapped to the Riemann zeta-function, whereas unrestricted partitions in arithmetic progressions can be mapped to the Dirichlet $L$-function or to the Hurwitz zeta-function. In the case of prime partitions, which are restrictive, the mapping goes to the prime zeta function. We therefore leverage the analytic properties of the zeta functions including their convergence, the location of their zeros, their analytic continuations, their singularities, as well as residues at the poles. However, in our case of interest, we will be dealing with (products of) prime zeta functions so we will \textit{also} have to account for logarithmic singularities as well as branch cuts.

The non-principal major arcs require an elaborate setup, namely leveraging the distribution of the set $\Nset$ in residue classes and our best weapon in this case is the Siegel-Walfisz theorem \cite{siegel, walfisz}.

Lastly, the minor arcs are by far the most difficult to bound. To get an idea of their contribution one needs specific technology for Weyl sums of the form
\begin{align*}
S_{\{\Nset,w\}}(\beta,y) = \sum_{\substack{n \le y \\ n \in \Nset}} w(n)\e(\beta n) \quad \textnormal{with} \quad \e(x) = \exp(2 \pi i x).
\end{align*}
For instance, in the case of restricted partitions over primes, one will need Vinogradov's bound $\sum_{\substack{p \le y}} \e(\beta p) \ll (yq^{-\frac{1}{2}}+y^{\frac{4}{5}} + y^{\frac{1}{2}} q^{\frac{1}{2}} )(\log 2y)^3$, where the sum is taken over primes, provided that $|\beta - \frac{a}{q}| \le \frac{1}{q^2}$ and $(a,q)=1$. For prime powers, Gafni \cite{gafniprimepowers} resorted to a bound by Kawada and Wooley \cite{kawadawooley}. It is worth remarking that the bound from Kawada and Wooley is given in \cite[Lemma 3.3]{kawadawooley} in dyadic form as
\begin{align*}
\sum_{\substack{P < p < 2P \\ p \in \mathbb{P}}} \e(\alpha p^k) \ll P^{1-2^{-k-1}+\varepsilon} + \frac{q^\varepsilon w_k(q)^{1/2} P(\log P)^4}{(1+P^k |\alpha-a/q|)^{1/2}},
\end{align*}
for $P \ge 2$ and where $w_k(q)$ is defined by
\begin{align*}
w_k(p^{uk+v})=
\begin{cases}
kp^{-u-1/2} & \mbox{ when $u \ge 0$ and $v=1$,}  \nonumber \\
p^{-u-1} & \mbox{ when $u \ge 0$ and $2 \le v \le k$.}
\end{cases}
\end{align*}
This bound is proved for integers $k$ such that $k \ge 4$. However, in \cite{kumchevalmostprimes} Kumchev explains that Wooley settled the case $k=3$ in \cite{wooley}. Moreover, one could use the work from Ghosh \cite{ghosh} to settle the case $k=2$, and lastly the case $k=1$ is due to Vinogradov \cite{vinogradov} and greatly simplified by Vaughan \cite{vaughanidentity}. In our case of semiprimes we will need a more elaborated version of Vinogradov's lemma for $\sum_{\substack{p_1 p_2 \le y}} \e(\beta p_1 p_2)$ and where the sum is taken over primes $p_1$ and $p_2$, namely
\begin{align*}
\sum_{p_1 p_2 \le X} \e (\beta p_1 p_2) \ll \frac{X}{q^{\frac{1}{6}}} (\log X)^{\frac{7}{3}} + X^{\frac{16}{17}} (\log X)^{\frac{39}{17}} + X^{\frac{7}{8}}q^{\frac{1}{8}} (\log X)^{\frac{9}{4}},
\end{align*}
where $\beta$ satisfies the above mentioned Diophantine conditions. We also remark that Lemmas 3.1, 3.2 and 3.3 from \cite{kumchevalmostprimes} do not work as suitably as the bilinear form supplied in Lemma 13.8 from of Iwaniec and Kowalski \cite{iwanieckowalski}. At any rate, the technology to bound the Weyl sums is arguably the most critical part of a successful application of the Hardy-Littlewood circle method.

Summarizing, only by having a satisfactory understanding of these three components for a given set $\Nset$ along with its weights $w(n)$, can we obtain the asymptotics for $\mathfrak{p}_{\{\Nset, w\}}(n)$. As argued in \cite{gafniprimepowers}, it is exceedingly rare to find a set $\Nset$, let alone with weights $w$, for which we have a good picture of these three components. As we shall see in the next sections, the Dirichlet series and its associated contour integral while technically difficult is in fact the most straightforward part, whereas the Siegel-Walfisz theorem will have to be applied many times in intricate ways, alongside the prime number theorem with the standard zero-free region, to get the bound of the non-principal major arcs. Finally, the machinery to deal with our associated Weyl sum will be supplied in Theorem \ref{doublevinogradov}.
%%%%%%%%%%%%%%%%%%%%%%%%%%%%%%%%%%%%%%%%%%%%%%%%%%%%%%%%%%%%%%%%%%%%%%

\section{Generating functions and arcs used}
\label{sec_st_up_arcs}

In this section, we specify the generating functions we are working with as well as the arcs that we shall use.

\subsection{Generating functions}
\label{sec:Gen_function_used}

We will work with the generating function
\begin{align}
\Psi_{\lambda,\mu}(z)
:=
\sum_{n=0}^\infty \partition_{\lambda,\mu}(n) z^n
:=
\exp\left( \Phi_{\lambda,\mu}(z)\right)
\label{eq:def_Psi_lambda_mu}
\end{align}
with $\lambda,\mu\in \R$ and $\Phi_{\lambda,\mu}(z) =\lambda\Phi_{\mathbb{P}_2}(z)  + \mu \Phi_{\mathbb{P}^2} (z)$ and where
\begin{align*}
\Phi_{\mathbb{P}_2} (z)
:= 
\sum_{j=1}^\infty \frac{1}{j} \sum_{p_1 \in \Pri} \sum_{p_2 \in \Pri} z^{(p_1 p_2)j}
\quad \text{and} \quad
\Phi_{\mathbb{P}^2} (z)
:= 
\sum_{j=1}^\infty \frac{1}{j} \sum_{p \in \Pri} z^{p^2j}.
\end{align*}
The function $\Psi_{\lambda,\mu}(z)$ in \eqref{eq:def_Psi_lambda_mu} covers all generating functions we require for the theorems in Section~\ref{sec:introduction}. Taking $\lambda = \mu = \frac{1}{2}$ leads to $\mathbb{P}_2$, taking $\lambda = \frac{1}{2}$ and $\mu = -\frac{1}{2}$ yields $\mathbb{P}_2^{\ne}$, and setting $\mu = 1$ and $\lambda = 0$ gives $\mathbb{P}_2^\sharp$.
In other words, we have
\begin{align}
\sum_{n=0}^\infty \partition_{\mathbb{P}_2}(n) z^n
=
\Psi_{\frac{1}{2},\frac{1}{2}}(z), \quad
\sum_{n=0}^\infty \partition_{\mathbb{P}_2^{\neq}}(n) z^n
=
\Psi_{\frac{1}{2},-\frac{1}{2}}(z), 
\quad \text{and} \quad
\sum_{n=0}^\infty \partition_{\mathbb{P}_2^\sharp}(n)z^n
=
\Psi_{1,0}(z).
\label{eq:generating_fkt_partition}
\end{align}	
Furthermore, $\Phi_{\mathbb{P}^2} (z)$ has been studied in detail in \cite{gafniprimepowers}. 
Thus we have to study $\Phi_{\mathbb{P}_2} (z)$ only to obtain the theorems in Section~\ref{sec:introduction}.
\begin{remark}
In order to connect $\Phi_{\mathbb{P}_2^\sharp}(z)$ with a weighted set, we could take
\begin{align*}
\Phi_{\{\mathbb{P}_2,w\}} (z) = \sum_{j=1}^\infty \sum_{m \in \mathbb{P}_2} \frac{w(m)}{j}z^{mj}.
\end{align*}
The weight function will be $w(n)=w_2(n)$ with
\begin{align*} %\label{weightsalmostprimes}
w_2(n) = \begin{cases}
2, \quad &\mbox{if $n$ is a product of two distinct primes}, \\
1, \quad &\mbox{if $n$ is the square of a prime,} \\
0, \quad &\mbox{otherwise.} 
\end{cases}
\end{align*}
In other words, $w_2(n)$ is the number of representations of $n$ in the form $p_1p_2$ where $p_1$ and $p_2$ are primes. Lastly, we define $\mathbf{1}_{\mathbb{P}_2}(n)$ be the indicator function of semiprimes.
\end{remark}
\subsection{Set up of the arcs}
Applying Cauchy's theorem to the generating function in \eqref{eq:generating_fkt_partition}, we get
\begin{align} \label{integralarcs}
\partition_{\lambda,\mu}(n)
&= 
\rho^{-n} \int_{-1/2}^{1/2} \Psi_{\lambda,\mu}(\rho \e(\alpha)) \e(-n \alpha) d\alpha \nonumber\\
&= 
\rho^{-n} \int_{-1/2}^{1/2} \exp\left(\lambda\Phi_{\mathbb{P}_2}(\rho \e(\alpha))+\mu\Phi_{\mathbb{P}^2}(\rho \e(\alpha))\right) \e(-n \alpha) d\alpha.
\end{align}
%Similar for $\partition_{\Oset(2)}(n)$ and $\partition_{\mathbb{P}_2^\sharp}$.
In order to prove our main theorems, we have to study the behaviour of $\Psi_{\lambda,\mu} (z)$ near the boundary of the unit disc.
An illustration of the values of $\Psi_{\lambda,\mu} (z)$ in the unit disc can be found in Figure~\ref{fig:illustration_Psi}.
We can see in this figure the largest values of $\Psi_{\lambda,\mu} (z)$ are near the point $1$, 
but also that $\Psi_{\lambda,\mu} (z)$ is large near $\e(a/q)$ with $a\in\Z$, $q\in\N$ and $q$ small.
Thus we have to carefully  split the integral in \eqref{strategy2}.
We define the major and minor arcs as follows. For real $A>18$ we set
\begin{align} \label{eq:def_delta_q}
\delta_q = q^{-1}X^{-1}(\log X)^A \quad \textnormal{and} \quad Q = (\log X)^A.
\end{align}
Moreover, for $1 \le a \le q \le Q$ with $(a,q)=1$ we define
\begin{align}
\mathfrak{M}(q,a) = \bigg(\frac{a}{q} - \delta_q, \frac{a}{q} + \delta_q\bigg).
\label{eq:def_M(q,a)}
\end{align}
The major $\mathfrak{M}$ and $\mathfrak{m}$ arcs will be defined by the following
\begin{align*}
\mathfrak{M} = \bigcup_{\substack{1 \le a \le q \le Q \\ (a,q)=1}} \mathfrak{M}(q,a) \quad \textnormal{and} \quad \mathfrak{m} = [-1/2,1/2) \backslash \mathfrak{M}.
\end{align*}
Next, in accordance to the strategy outlined in Section \ref{sec:strategy}, we divide the integral in \eqref{integralarcs} into three pieces:
\begin{itemize}
\item the principal major arc $\mathfrak{M}(1,0)$,
\item the non-principal major arcs $\mathfrak{M}(q,a)$ with $q>1$,
\item and the minor arcs $\mathfrak{m}$.
\end{itemize}
The main contribution to the integral comes from $\Phi_{\mathbb{P}_2}(z)$. 
Thus we do not have to take into account $\Phi_{\mathbb{P}^2}(z)$ for the choice of the arcs.
In fact, we bound the term $\Phi_{\mathbb{P}^2}(z)$ on the non-principal arcs trivially.
\begin{figure}[H] 
\includegraphics[scale=0.36]{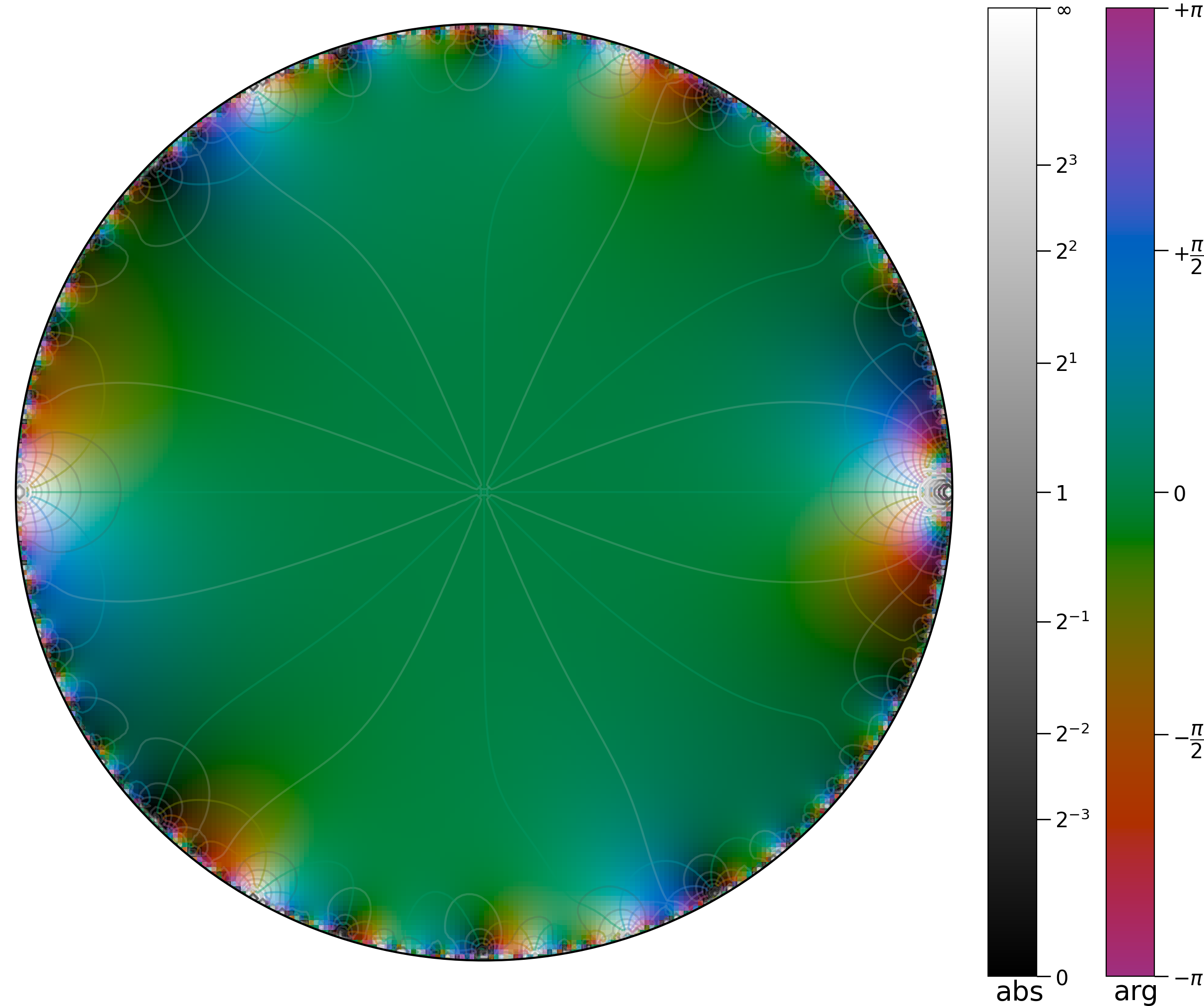} %0.54
\caption{Domain plot of a truncation of $\Phi(z)$ given by $\Phi_M(z)= \prod_{n=1}^M (\frac{1}{1-z^n})^{f(n)}$ where $f(n):= \mathbf{1}_{\mathbb{P}_2}(n)$ with $M=10^4$ and absolute value and argument hue.}
\label{fig:illustration_Psi}
\label{fig:generatingfunction}
\end{figure}
The main term will be dictated entirely by the principal major arcs and this main term will be extracted in Sections \ref{sec:principalarcs} and \ref{sec:mainresult}. The bound for the denominators in the major arcs is limited by the scope of the Siegel-Walfisz theorem, see Section \ref{sec:nonprincipalarcs}. 
We have chosen the exponent of $\log X$ to be greater than 18 in order to yield a satisfactory bound on the minor arcs. 
We shall  show the details of the derivation in Section \ref{sec:minorarcs}.
%
%%%%%%%%%%%%%%%%%%%%%%%%%%%%%%%%%%%%%%%%%%%%%%%%%%%%%%%%%%%%%%%%%%%%%%
\section{The principal major arcs} \label{sec:principalarcs}
As discussed in Section \ref{sec:introduction} one can compute the main term arising from the principal major arc $\mathfrak{M}(1,0)$ with the traditional method of contour integration or setting up an appropriate prime number theorem with a satisfactory error term.
\begin{theorem} \label{maintermstheorem}
Let $\rho = e^{-1/X}$. Then for any $m \in \Z_{\ge 0}$, we have
\begin{align} \label{maintermtheorem1}
\bigg(\rho \frac{d}{d\rho} \bigg)^m \Phi_{\mathbb{P}_2}(\rho) 
&= 
2\frac{\zeta(2) \Gamma(m+1) X^{m+1}}{\log X}(M + \log \log X ) \bigg(1+O\bigg(\frac{1}{\log X}\bigg)\bigg) 
\end{align}
as well as
\begin{align} \label{maintermtheorem2}
\Phi_{\mathbb{P}_2}^{(m)}(\rho) = 2\frac{\zeta(2) \Gamma(m+1) X^{m+1}}{\log X}(M + \log \log X ) \bigg(1+O\bigg(\frac{1}{\log X}\bigg)\bigg)
\end{align}
as $\rho \to 1^-$. Here $M$ is the Meissel-Mertens constant $M = \gamma + \sum_p [\log(1-\frac{1}{p})+\frac{1}{p}] \approx 0.26149721 \ldots$
\end{theorem}
%%%%%%%%%%%%%%%%%%%%%%%%%%%%%%%%%%%%%%%%%%%%%%%%%%%%%%%%%%%%%%%%%%%%%%
\subsection{The method of contour integration}
The first method of proof is shown here. 
\begin{proof}[Proof of Theorem \eqref{maintermstheorem}]
Recall that we have
\begin{align*}
\Phi_{\mathbb{P}_2}(\rho) = \sum_{j=1}^\infty \frac{1}{j} \sum_{p_1} \sum_{p_2} \rho^{p_1 p_2 j}. 
\end{align*}
Inserting the definition of $\rho$ yields
\begin{align*}
\bigg(\rho \frac{d}{d\rho}\bigg)^m \Phi_{\mathbb{P}_2}(\rho) = \sum_{j=1}^\infty j^{m-1} \sum_{p_1} p_1^m \sum_{p_2} p_2^m e^{-p_1 p_2 j /X}.
\end{align*}
We now employ the Mellin transform of the Gamma function so that
\begin{align} \label{auxintegral}
\bigg(\rho \frac{d}{d\rho}\bigg)^m \Phi_{\mathbb{P}_2}(\rho) &= \frac{1}{2 \pi i} \int_{(c)} X^s \bigg(\sum_{p_1} \frac{1}{p_1^{s-m}}\bigg) \bigg(\sum_{p_2} \frac{1}{p_2^{s-m}}\bigg) \bigg(\sum_{j=1}^\infty \frac{1}{j^{s+1-m}} \bigg) \Gamma(s) ds \nonumber \\
&= \frac{1}{2 \pi i} \int_{(c)} X^s (\zeta_{\mathcal{P}}(s-m))^2 \zeta(s+1-m) \Gamma(s) ds, 
\end{align}
since $c > m+1$ and where $\zeta_{\mathcal{P}}(s)$ is the prime zeta function defined in \eqref{primezetamoebius} for $\real(s)>1$. The next step is to note that $\zeta_{\mathcal{P}}(s) = \log \zeta(s) - \mathcal{D}(s)$ where
\begin{align*}
\mathcal{D}(s) = \sum_{j \ge 2}\frac{1}{j} \sum_p \frac{1}{p^{js}}.
\end{align*}
For any $\delta > 0$ we have that $\mathcal{D}(s)$ converges absolutely and uniformly for $\real(s) \ge \frac{1}{2} + \delta$. If we make the replacement $(\zeta_{\mathcal{P}}(s-m))^2 = (\log \zeta(s-m))^2 - 2 \mathcal{D}(s-m) \log \zeta(s-m) + (\mathcal{D}(s-m))^2$ in \eqref{auxintegral}, then we could move the contour of integration to the line $\real(s) = c_0$ for any $c_0 > m + \tfrac{1}{2}$. The task at hand is therefore to compute the contribution from these three parts. We start with the first, and most difficult, case, that of $(\log \zeta(s-m))^2$, i.e.
\begin{align} \label{squareintegral}
\Omega(m, X) := \frac{1}{2 \pi i} \int_{(c)} X^s (\log \zeta(s-m))^2 \zeta(s+1-m) \Gamma(s) ds .
\end{align}
We begin by noting that the integral is analytic in the zero-free region for $\zeta(s-m)$ except for a logarithmic singularity at $s=m+1$. If we choose $T = \exp (\sqrt{\log X})$, then the integral can be truncated at height $T$ with an acceptable error term. The remaining part of the integral can be shifted to the left of the line $\real(s) = m + 1-\tfrac{c}{\log T}$ where $c$ is a suitable positive constant, except for a keyhole contour around the essential singularity at $s=m+1$. This keyhole contour runs counter clockwise along the top and the bottom of the branch cut located at $\{s= \sigma \, : \, \sigma \le m+1\}$, as in Figure \ref{fig:integrationcontour}.
Setting $\Xi$ to denote this contour and since the integrand is analytic in $\Xi$ we have, by Cauchy's theorem, that
\begin{align} \label{cauchy}
0 &= \frac{1}{2 \pi i} \oint_{\Xi} X^s (\log \zeta(s-m))^2 \zeta(s+1-m) \Gamma(s) ds \nonumber \\
&= \frac{1}{2 \pi i}\bigg(\int_{\Xi_1} + \int_{\Xi_2} + \cdots + \int_{\Xi_8}\bigg) X^s (\log \zeta(s-m))^2 \zeta(s+1-m) \Gamma(s) ds.
\end{align}
Ignoring for the moment the top and bottom parts, all the remaining segments and curves of the contour are well-controlled and will only contribute to the error term.
\begin{figure}[H]
\includegraphics[scale=0.293]{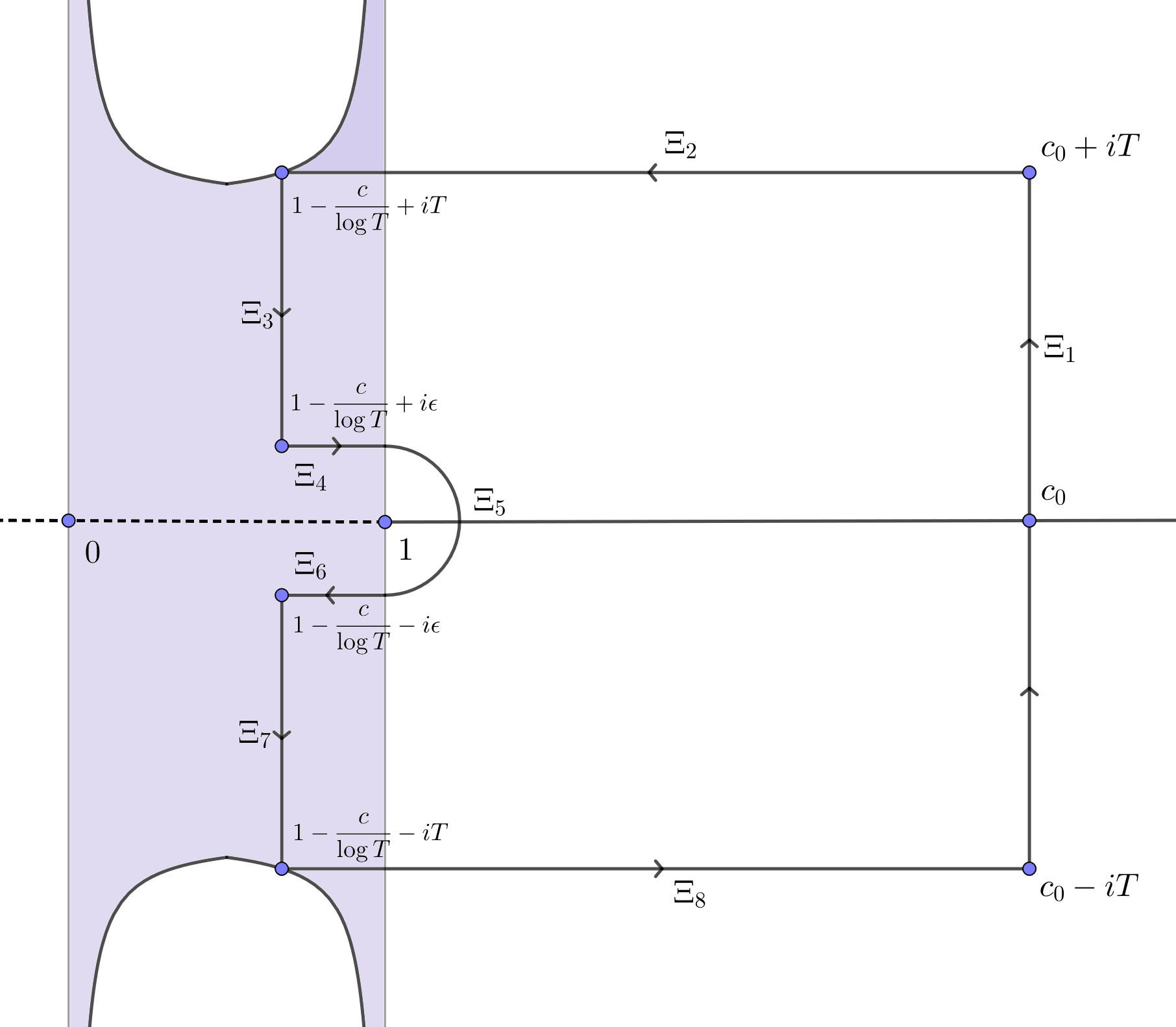}
\caption{The contour of integration $\Xi$ alongside the zero-free region of $\zeta(s)$ in blue.}\label{fig:integrationcontour}
\end{figure}
Let us now concentrate on the top part of the branch cut which is given by 
\begin{align*}
\Omega_{\textnormal{top}}(m, X) = \frac{1}{2 \pi i}\int_{m+1-\tfrac{c}{\log T}}^{m+1} X^s (\log \zeta(s-m))^2 \zeta(s+1-m) \Gamma(s) ds.
\end{align*}
We make the change of variables $s=m+1-u$ and realize that on the top branch of the cut we are dealing with $u \to u + i \varepsilon$ for $\varepsilon >0$. In this case, our integral becomes 
\begin{align*}
\frac{1}{2 \pi i}\int_{0}^{\tfrac{c}{\log T}} X^{m+1-(u+i \varepsilon)} (\log \zeta(1-(u+i\varepsilon)))^2 \zeta(2 - (u+i \varepsilon)) \Gamma(m+1 - (u+i \varepsilon)) du.
\end{align*}
If we use the identity
\begin{align*}
\log \zeta(s) = -\log (s-1) + g(s) \quad \textnormal{where} \quad g(s) = \log ((s-1) \zeta(s)),
\end{align*}
as well as $\log(a+ib) = \tfrac{1}{2}\log(a^2+b^2)+ i \theta$ where $a+ib = \sqrt{a^2+b^2}e^{i \theta}$ for $a,b \in \R$, then 
the above becomes
\begin{align*}
\frac{1}{2 \pi i}\int_{0}^{\tfrac{c}{\log T}} X^{m+1-(u+i \varepsilon)} &\bigg(-\frac{1}{2} \log(u^2 + \varepsilon^2) + i \theta + g (1-(u+i \varepsilon)) \bigg)^2 \nonumber \\
& \times \zeta(2 - (u+i \varepsilon)) \Gamma(m+1 - (u+i \varepsilon)) du.
\end{align*}
Letting $\varepsilon \to 0$ and realizing that on the top of the branch we have $\theta = \pi$ leads us to 
\begin{align*}
\Omega_{\textnormal{top}}(m, X) = \frac{1}{2 \pi i}\int_{0}^{\tfrac{c}{\log T}} X^{m+1-u} (- \log u + i \pi + g (1-u))^2 \zeta(2 - u) \Gamma(m+1 - u) du.
\end{align*}
A similar reasoning on the bottom part of the branch cut, where $\theta = -\pi$, shows that 
\begin{align*}
\Omega_{\textnormal{bottom}}(m, X) = \frac{1}{2 \pi i}\int_{0}^{\tfrac{c}{\log T}} X^{m+1-u} (- \log u - i \pi + g (1-u))^2 \zeta(2 - u) \Gamma(m+1 - u) du.
\end{align*}
Therefore the contribution along the two segments of the branch cut is given by 
\begin{align*}
\Omega(m, X) = \frac{1}{2 \pi i}\int_{0}^{\tfrac{c}{\log T}} X^{m+1-u} \vartheta(u) \zeta(2 - u) \Gamma(m+1 - u) du,
\end{align*}
where we have defined 
\begin{align*}
\vartheta(u) = (- \log u + i \pi + g (1-u))^2 - (- \log u - i \pi + g (1-u))^2 = 4 \pi i (-\log u + g(1-u)).
\end{align*}
Inserting this into our integral of interest yields
\begin{align*}
\Omega(m, X) = \frac{1}{2 \pi i} \int_0^{\tfrac{c}{\log T}} 4 \pi i (-\log u + g(1-u)) X^{m+1-u}\zeta(2-u)\Gamma(m+1-u) du.
\end{align*}
For any $m$ we have that 
\begin{align*}
\zeta(2-u) \Gamma(m+1-u) = \zeta(2) \Gamma(m+1) + O(u) \quad \textnormal{and} \quad g(1-u)=-\gamma u +O(u)^2
\end{align*}
uniformly for $u \in [0,1/2]$ and where $\gamma$ is the Euler constant. Thus the contribution from the cuts is given by
\begin{align*}
\Omega(m, X) = 2 \zeta(2) \Gamma(m+1) X^{m+1}\bigg(-\int_0^{\tfrac{c}{\log T}} X^{-u} \log u du + O \bigg(X^{m+1} \int_0^{\frac{c}{\log T}}X^{-u}udu \bigg)\bigg).
\end{align*}
To compute this integral we invoke the formula
\begin{align*}
\int_0^{c/\log T} X^{-u} \log u du = -\frac{\gamma + \log \log X + \Gamma(0,\frac{c \log X}{\log T}) + X^{-\frac{c}{\log T}} \log \frac{c}{\log T}}{\log X},
\end{align*}
provided that $\real(\log X)>0$ and where $\Gamma(a,z) := \int_z^\infty t^{a-1 }e^{-t} dt$ is the incomplete gamma function. This can be shown by considering (see e.g. \cite[Equation (6.2)]{przz} or \cite[$\mathsection$3.3)]{krz})
\begin{align*}
I(a,b) := \int_0^a b^{-u}\log u du \quad \textnormal{and} \quad \log x = - \frac{1}{2\pi i} \oint \frac{1}{x^z}\frac{dz}{z^2},
\end{align*}
where the integral is taken over a small circle around the origin. We use this to swap the order of integration so that
\begin{align*}
I(a,b) &= - \int_0^a b^{-u} \frac{1}{2\pi i}\oint \frac{1}{u^z} \frac{dz}{z^2}du =- \frac{1}{2\pi i}\oint\int_0^a b^{-u} \frac{du}{u^z} \frac{dz}{z^2} \nonumber \\
&= - \frac{1}{2\pi i}\oint \frac{a^{-z}}{\log b}(-a \operatorname{E}(z,a \log b)\log b+\Gamma(1-z)(a\log b)^z) \frac{dz}{z^2} \nonumber \\
&= -\frac{1}{\log b}(\gamma+ \Gamma(0,a\log b)+b^{-a}\log a + \log \log b),
\end{align*}
by a direct residue calculus computation and where $\operatorname{E}(n,z):= \int_1^\infty e^{-zt}/t^n dt$ is the exponential integral function. Thus we arrive at
\begin{align} \label{omegao1}
\Omega(m, X) &= 2 \zeta(2) \Gamma(m+1) X^{m+1} \bigg(\frac{\gamma + \log \log X + \Gamma(0,\frac{c \log X}{\log T}) + X^{-\frac{c}{\log T}} \log \frac{c}{\log T}}{\log X}\bigg) \nonumber \\
& \quad + O\bigg(X^{m+1} \frac{1-X^{-\frac{c}{\log T}}(1+\frac{c \log X}{\log T})}{\log^2 X}\bigg) \nonumber \\
&= 2 \zeta(2) \Gamma(m+1) X^{m+1} \bigg(\frac{\gamma + \log \log X}{\log X} + \frac{\Gamma(0, c \sqrt{ \log X})}{\log X} + \frac{ e^{- \sqrt{\log X}} \log \frac{c}{\sqrt{\log X}}}{\log X}\bigg) \nonumber \\
& \quad + O\bigg(X^{m+1} \frac{1-e^{-c \sqrt{\log X}}(1+ c \sqrt{\log X})}{\log^2 X}\bigg) 
\end{align}
since $T= \exp(\sqrt{\log X})$. Employing \eqref{cauchy} we end up with
\begin{align*}
\Omega(m, X) = \frac{2 \zeta(2) \Gamma(m+1) X^{m+1} (\gamma+\log \log X)}{\log X} \bigg( 1 + O \bigg(\frac{1}{\log X}\bigg)\bigg),
\end{align*}
where we have used that
\begin{align*}
\Gamma(0, c \sqrt{\log X}) = e^{-c \sqrt{\log X} + O(\frac{1}{X})} \bigg(\frac{-1+c \sqrt{\log X}}{c^2 \log X} + O\bigg(\frac{1}{X}\bigg)\bigg)
\end{align*}
as $X\to \infty$ for $c>0$. We now need to compute the piece coming from $-2 \mathcal{D}(s-m) \log \zeta(s-m)$ which is given by
\begin{align} \label{squareintegralD}
\tilde\Omega(m, X) := \frac{1}{2 \pi i} \int_{(c)} X^s (-2\mathcal{D}(s-m)\log \zeta(s-m)) \zeta(s+1-m) \Gamma(s) ds .
\end{align}
The technique is nearly identical to the one we just described except with $\tilde\vartheta(u)= (- \log u + i \pi + g (1-u)) - (- \log u - i \pi + g (1-u)) = 2 \pi i $ and we will end up having
\begin{align*}
\tilde\Omega(m, X) &= -2\frac{1}{2 \pi i}\int_{0}^{\tfrac{c}{\log T}} \tilde\vartheta(u) X^{m+1-u} \mathcal{D}(1-u) \zeta(2 - u) \Gamma(m+1 - u) du \nonumber \\
&= -\frac{2 \mathcal{D}(1) \zeta(2) \Gamma(m+1) X^{m+1}}{\log X}\bigg( 1 + O \bigg(\frac{1}{\log X}\bigg)\bigg),
\end{align*}
where $\mathcal{D}(1) = \mathcal{D}$ with $\mathcal{D} = \sum_{j \ge 2} \frac{1}{j} \sum_p \frac{1}{p^j} =  -\sum_{n=2}^\infty \frac{\mu(n)}{n} \log \zeta(n) \approx 0.315718452 \ldots$ being Fr\"{o}berg's constant \cite{froeberg}. One can show that $\gamma+\mathcal{D} = M$ in \cite[$\mathsection$6]{nathanson}. Lastly, the contribution involving the $(\mathcal{D}(s-m))^2$ piece is
\begin{align} \label{boundD2}
\tilde{\tilde\Omega}(m, X) := \frac{1}{2 \pi i} \int_{(c_0)} X^s (\mathcal{D}(s-m))^2 \zeta(s+1-m) \Gamma(s) ds \ll X^{c_0}.
\end{align}
We now choose $c_0 = m + \frac{3}{4}$, say, in \eqref{boundD2} so that \eqref{maintermtheorem1} follows. We can close the statement of the lemma by a standard argument. First the case $m=0$ is immediate. We can use induction on $m$ to write
\begin{align*}
\rho^m \Phi_{\mathbb{P}_2}^{(m)}(\rho) = \sum_{i=1}^m c_{i,m} \bigg(\rho \frac{d}{d\rho}\bigg)^i \Phi_{\mathbb{P}_2}(\rho),
\end{align*}
where $c_{i,m}$ are reals with $c_{m,m}=1$. The last step is to use the fact that $\rho = 1 +O(X^{-1})$ and the result now follows.
\end{proof}
\subsection{Moment method for the major arcs}
We now show that this main term can be obtained by the method of moments and a tailored prime number theorem. We choose to work with the weakest error term of the prime number theorem, i.e. knowing that $\zeta(s) \ne 0$ for $\real(s)=1$ in order to showcase the fewer resources needed to partially achieve our goal. The proof we present can be adapted to include a more refined error term.

\begin{lemma}
\label{lem:eq:integral_with_gamma3}
Let $L>1$, $\lambda>0$, $a, b\geq 0$ and $M \geq 0$ be given. One has that
\begin{align}
\int_{L}^\infty e^{-at} t^{\lambda}\frac{\log\log t +M }{(\log t)^b} \, dt
=
\frac{\Gamma(\lambda+1)(\log\log(1/a)+M)}{a^{\lambda+1}(\log (1/a))^b}\left(1 + O\left(\frac{\log\log(1/a)}{\log(1/a)}\right)\right),
\label{eq:integral_with_gamma3}
\end{align}
as $a \to 0^+$.
\end{lemma}
The proof of this result will be very similar to the proof of Lemma \ref{lem:dirk'slemma} which will be shown in Section \ref{sec:nonprincipalarcs} and thus we postpone it for later.
\begin{corollary} 
\label{cor:eq:integral_with_gamma3}
Let $f:[L, \infty) \to\R$ with $L>1$ be a function such that $f(t) =o(t)$ as $t\to\infty$.
We then have as $a\to0$ for $a>0$
\begin{align}
\int_{L}^\infty e^{-at} f(t)\frac{\log\log t}{\log t} \, dt
=
o\left(\frac{\log\log(1/a)}{a^{2}\log(1/a)}\right).
\label{eq:integral_with_gamma3.5}
\end{align}
\end{corollary}
\begin{proof}
The proof is almost the same as the proof of Lemma~\ref{lem:dirk'slemma}.
Indeed, there exists a $C>0$ such that $|f(t)|\leq Ct$  for all $t$. 
Thus we can use for the integrals over $[L,d]$ and $[u,\infty]$ exactly the same bounds as in the proof of Lemma~\ref{lem:dirk'slemma}.
Further, since $u,d\to\infty$ as $a\to 0$ and $f(t)=o(1)$, there exists for each $c>0$ an $a_0$ such that $|f(t)|\leq ct$ for all $t\in [d,u]$ and $a\geq a_0$. 
Combining this with the computation of the integral over $[d,u]$ completes the proof.
\end{proof}
Now we state and prove a prime number theorem for semiprimes.
\begin{lemma}
\label{lem:PNT_two_primes}
We have 
\begin{align*}
\pi_2^*(x) := \sum_{\substack{p_1p_2 \leq x \\ p_1, p_2 \in \mathbb{P}}} 1 
\sim
2x\frac{\log\log x}{\log x}.
\end{align*}
as $x\to\infty$.
\end{lemma}
\begin{proof}
We denote by $\pi(x) = \sum_{p \le x} 1$ the usual prime counting function.
Then we get
\begin{align*}
\pi_2^*(x)
&=
\sum_{p_1\leq x^{1/2}} \sum_{p_2\leq x/p_1} 1 + \sum_{p_2\leq x^{1/2}}	\sum_{p_1\leq x/p_2}  1 - \sum_{p_1\leq x^{1/2}} 1 \sum_{p_2\leq x^{1/2}} 1 \nonumber\\
&=
2\sum_{p\leq x^{1/2}} \pi(x/p) - (\pi(x^{1/2}))^2.
\end{align*}
Using that $\pi(x)= \frac{x}{\log x}(1+o(1))$ and  $x^{1/2}\leq x/p\leq x$, we obtain 
\begin{align}
\pi_2^*(x) 
&=
2 \bigg(\sum_{p\leq x^{1/2}} \frac{x}{p\log(x/p)}\bigg)(1+o(1))
+O\left(\frac{x}{(\log x)^2}\right).
\label{eq:sum_in_proof_sum_of_primes2}
\end{align}
Mertens' theorem states that $\sum_{p\leq x} \frac{\log p}{p} = \log x + O(1)$, see \cite[$\mathsection$6]{nathanson}.
Combining this with Abel's summation formula with $f(t) = \frac{1}{\log(x/t) \log t}$ gives
\begin{align}
\sum_{p\leq x^{1/2}}  \frac{1}{p\log(x/p)}
&=
\sum_{p\leq x^{1/2}}  \frac{\log p}{p}f(p)
=
f(x^{1/2})\bigg(\sum_{p\leq x^{1/2}}  \frac{\log p}{p} \bigg) 
-
\int_2^{x^{1/2}} f'(t)\bigg(\sum_{p=2}^{t}  \frac{\log p}{p}\bigg)  dt\nonumber\\
&=
\int_2^{x^{1/2}} (\log t +O(1))   \frac{\log x -2\log t}{t(\log t)^2(\log(x/t))^2} dt
+\frac{2}{\log x}+O\left(\frac{1}{(\log x)^2}\right)\nonumber\\
&=
\int_2^{x^{1/2}} \frac{\log x}{t(\log t)(\log(x/t))^2}\,dt
-
2\int_2^{x^{1/2}}  \frac{1}{t(\log(x/t))^2}\,dt\nonumber\\
& 
\quad + O\bigg(\int_2^{x^{1/2}}  \frac{\log x -2\log t}{t(\log t)\log^2(x/t)} dt \bigg)
+O\left(\frac{1}{\log x}\right).
\label{eq:sum_in_proof_sum_of_primes4}
\end{align}
We now look at the three integrals in \eqref{eq:sum_in_proof_sum_of_primes4} separately.
Since $\frac{1}{2}\log t \leq \log(t/x)\leq \log t$, for the last integral we have
\begin{align}
\int_2^{x^{1/2}}   \frac{\log x -2\log t}{t (\log t)^2\log^2(x/t)} dt 
\ll
\frac{1}{\log x}\int_2^{x^{1/2}}  \frac{dt}{t(\log t)^2}  
\ll \frac{1}{(\log x)^2}.
\label{eq:sum_in_proof_sum_of_primes66}
\end{align}
Further,
\begin{align}
\int_2^{x^{1/2}}  \frac{dt}{t\log^2(x/t)}
%=
%\left[ \frac{1}{\log(x/t)}\right]_{t=2}^{x^{1/2}}
=
\frac{1}{\log x} + O\left(\frac{1}{(\log x)^2}\right).
\label{eq:sum_in_proof_sum_of_primes63}
\end{align}
Using that $1/(1-r)^2 \leq 1+4r$ for $|r|\leq 1/2$, we get
\begin{align}
\int_2^{x^{1/2}} \frac{\log x}{t\log(t)\log^2(x/t)} dt
&=
\frac{1}{\log x}
\int_2^{x^{1/2}} \frac{dt}{t\log(t)(1-\frac{\log t}{\log x})^2} \nonumber\\
&\leq 
\frac{1}{\log x}
\int_2^{x^{1/2}} \frac{dt}{t\log t}  + 
4\int_2^{x^{1/2}} \frac{dt}{t(\log t)^2} =
\frac{\log\log x}{\log x}
+
O\left( \frac{1}{\log x}  \right).
\label{eq:sum_in_proof_sum_of_primes6}
\end{align}
Inserting \eqref{eq:sum_in_proof_sum_of_primes66}, \eqref{eq:sum_in_proof_sum_of_primes63} and \eqref{eq:sum_in_proof_sum_of_primes6} in to \eqref{eq:sum_in_proof_sum_of_primes4} gives
\begin{align}
\sum_{p=2}^{x^{1/2}}  \frac{1}{p\log(x/p)}
=
\frac{\log\log x}{\log x}
+
O\left( \frac{1}{\log x}  \right).
\label{eq:sum_in_proof_sum_of_primes8}
\end{align}
Inserting \eqref{eq:sum_in_proof_sum_of_primes8} into \eqref{eq:sum_in_proof_sum_of_primes2} completes the proof.
\end{proof}

We have for $z\in\C$ with that
\begin{align*}
\Phi_{\mathbb{P}_2}(z)
=
\sum_{j=1}^\infty
\frac{1}{j}
\sum_{p_1,p_2 \in \mathbb{P}}
z^{p_1p_2j}
\end{align*}
We write $z=\rho \e(\alpha)$ with $0\leq \rho<1$, $\alpha\in\R$.
Further, we set $X= (\log \frac{1}{\rho})^{-1}$ and thus $\rho =e^{-1/X}$. 
In particular, we have $X\to\infty$ if $\rho\to 1^-$. 
We thus can write $\Phi_{\mathbb{P}_2}(z)$ as 
\begin{align}
\Phi_{\mathbb{P}_2}(\rho \e(\alpha))
=
\sum_{j=1}^\infty
\frac{1}{j}
\sum_{p_1,p_2 \in \mathbb{P}}
\rho^{-p_1p_2j}\e(jp_1p_2\alpha)
=
\sum_{j=1}^\infty
\frac{1}{j}
\sum_{p_1,p_2 \in \mathbb{P}}
e^{-p_1p_2j/X}\e(jp_1p_2\alpha).
\label{eq:Phi_with_X_and_alpha_two_primes}
\end{align}
We will also need the behaviour of the derivatives of  $\Phi$.
We have for $m\in\N_0$ that
\begin{align}
\left(\rho\frac{\partial}{\partial \rho}\right)^m 
\Phi_{\mathbb{P}_2}(\rho \e(\alpha))
=
\sum_{j=1}^\infty
\frac{1}{j}
\sum_{p_1,p_2 \in \mathbb{P}}(p_1p_2j)^m
e^{-p_1p_2j/X}\e(jp_1p_2\alpha).
\label{eq:Phi_with_X_and_alpha_mth_derivative_with_two_primes}
\end{align}
Furthermore, for all $b>0$ and $m\in\N_0$ that
\begin{align}
(bj)^m e^{-bj/X}
=
\int_{b}^\infty   
(j^mt^m - mt^{m-1}j^{m-1}X) \frac{j}{X} e^{-\frac{j}{X}t}\, dt.
\label{eq:exp_as_integral}
\end{align}
Indeed, partial integration implies that
\begin{align*}
\int_{b}^\infty   
j^mt^m \frac{j}{X} e^{-\frac{j}{X}t}\, dt
%&=
%\left[-j^mt^m e^{-\frac{j}{X}t}\right]_{t=b}^\infty + 
%\int_{b}^\infty   j^m mt^{m-1} e^{-\frac{j}{X}t}\, dt\\
&=
j^mb^m e^{-\frac{j}{X}b}
+
\int_{b}^\infty   (mj^{m-1} t^{m-1} X)\frac{j}{X} e^{-\frac{j}{X}t}  dt.
\end{align*}
Combining \eqref{eq:Phi_with_X_and_alpha_mth_derivative_with_two_primes} and \ref{eq:exp_as_integral} gives
\begin{align}
\left(\frac{\partial}{\partial \rho}\right)^m \Phi_{\mathbb{P}_2}(\rho \e(\alpha))
&=
\sum_{j=1}^\infty
\frac{1}{j}
\sum_{p_1,p_2} \e(jp_1p_2\alpha)
\int_{p_1p_2}^\infty (j^{m} t^{m}-mj^{m-1} t^{m-1} X) \frac{j}{X} e^{-\frac{j}{X}t}\, dt\nonumber\\
&=
\sum_{j=1}^\infty
\frac{1}{j}
\int_{2^2}^\infty (j^{m} t^{m}-mj^{m-1} t^{m-1} X) \frac{j}{X} e^{-\frac{j}{X}t} \sum_{p_1p_2\leq t} \e(jp_1p_2\alpha)  \, dt.
\label{eq:Phi_with_X_and_alpha_mth_derivative_with_integral_two_primes}
\end{align}
Equipped with these tools we may now prove the main result.
\begin{lemma}
\label{lem:moments_of_Phi_two_primes}
We have for all $m\in\N_0$ that
\begin{align*}
\left(\rho\frac{\partial}{\partial \rho}\right)^m 
\Phi_{\mathbb{P}_2}(\rho)
=
2\frac{\zeta(2)\Gamma(m+1)X^{m+1} \log\log X}{\log X}\left(1 + o(1) \right).
\end{align*}
\end{lemma}
\begin{proof}
We can write \eqref{eq:Phi_with_X_and_alpha_mth_derivative_with_integral_two_primes} as
\begin{align}
\left(\rho\frac{\partial}{\partial \rho}\right)^m \Phi_{\mathbb{P}_2}(\rho)
&=
\sum_{j=1}^\infty \frac{1}{j}
\int_{2^{2}}^\infty (j^{m} t^{m}-mj^{m-1} t^{m-1} X) \frac{j}{X} e^{-\frac{j}{X}t} \pi_2^*(t)  \, dt.
\label{eq:phi_at_rho_with_derivatives_two_primes}
\end{align}
We first derive an upper bound for integral in \eqref{eq:phi_at_rho_with_derivatives_two_primes}.
The variable substitution $u=\frac{j}{X}t$ gives
\begin{align*}
\left|\int_{2^{2}}^\infty (j^{m} t^{m}-mj^{m-1} t^{m-1} X)  \frac{j}{X}e^{-\frac{j}{X}t} \pi_2^*(t) \, dt\right|
&\leq 
C\int_{0}^\infty ((jt)^m+mX(jt)^{m-1})  \frac{jt}{X} e^{-\frac{j}{X}t} dt\\
&\leq 
C\frac{X}{j} \int_{0}^\infty ((Xu)^m+mX(Xu)^{m-1})  u e^{-u}  du\\
&\leq 
C\frac{X^{m+1}}{j}\int_{0}^\infty (u^{m+1}+mu^{m})  e^{-u} du\\
&=
C\frac{X^{m+1}}{j} (\Gamma(m+2)+m\Gamma(m+1))
\ll	\frac{X^{m+1}}{j}.
\end{align*}
We now split the sum in \eqref{eq:phi_at_rho_with_derivatives_two_primes} into the cases $j>x^{1/2}$ and  $j\leq x^{1/2}$. 
Inserting the above bound and using that $\sum_{j>x^{1/2}}^\infty j^{-2} \ll x^{-1/2}$ gives
\begin{align}
\left(\frac{\partial}{\partial \rho}\right)^m \Phi_{\mathbb{P}_2}(\rho)
&=
\sum_{j\leq x^{1/2}}
\frac{1}{j}
\int_{2^{2}}^\infty (j^{m} t^{m}-mj^{m-1} t^{m-1} X) \frac{j}{X} e^{-\frac{j}{X}t} \pi_2^*(t)  \, dt
+
O(	X^{m+1/2}).
\label{eq:Phi_with_X_and_alpha_mth_derivative_with_integral_truncated_two_primes}
\end{align}
Next we compute the integrals in \eqref{eq:Phi_with_X_and_alpha_mth_derivative_with_integral_truncated_two_primes}.
Using Lemma~\ref{lem:eq:integral_with_gamma3} with $b=c=1$ and Corollary~\ref{cor:eq:integral_with_gamma3} yields
\begin{align}
\int_{2^{2}}^\infty j^{m} t^{m} \frac{j}{X} e^{-\frac{j}{X}t} \pi_2^*(t) dt
&=
\frac{2j^{m+1}}{X} \int_{2^{2}}^\infty t^{m+1} e^{-\frac{j}{X}t} \frac{\log\log t}{\log t}(1+o\left(1\right)) dt\nonumber\\
&=
\frac{2j^{m+1}}{X} \frac{\Gamma(m+2)\log\log(X/j)}{(j/X)^{m+2}\log(X/j)}(1 + o\left(1\right))\nonumber\\
&=
2\frac{X^{m+1}}{j}\frac{\Gamma(m+2)\log\log(X/j)}{\log(X/j)}(1 + o\left(1\right)).
\label{eq:lem:moments_of_Phi_one_of_the _integrals_two_primes}
\end{align}
The $o(1)$ term in the last two equations has to be interpreted $o(1)$ as $X/j\to\infty$.
Note that this $o(1)$ depends only on the ratio $X/j$, but not on $X$ and $j$. 
Inserting \eqref{eq:lem:moments_of_Phi_one_of_the _integrals_two_primes} with $m$ and $m-1$ into \eqref{eq:Phi_with_X_and_alpha_mth_derivative_with_integral_truncated_two_primes} leads us to
\begin{align}
\left(\frac{\partial}{\partial \rho}\right)^m \Phi_{\mathbb{P}_2}(\rho)
&=
\sum_{j\leq x^{1/2}}
2\frac{X^{m+1}}{j^{2}}\frac{\Gamma(m+2) -m\Gamma(m+1) }{\log(X/j)}\log\log(X/j)
(1 + o(1))
+
O(	X^{m+1/2})\nonumber\\
&=
2X^{m+1}\Gamma(m+1)\sum_{j\leq x^{1/2}}
\frac{1}{j^{2}}\frac{ \log\log(X/j)}{\log(X/j)}(1 + o(1))
+
O(	X^{m+1/2}).
\label{eq:last_step_moments}
\end{align}
Using that $j\leq X^{1/2}$, we get	
\begin{align*}
\frac{\log\log(X/j)}{\log(X/j)} 
%			=
%			\frac{\log\log(X) +O(\frac{\log(j)}{\log(X)}) }{\log(X) (1-\frac{\log(j)}{\log(X)})}
=
\frac{\log\log X}{\log X} +O\left(\frac{\log\log X}{(\log X)^2}\right).
\end{align*}
Inserting this into \eqref{eq:last_step_moments} and that $\sum_{j> x^{1/2}}j^{-2} \ll X^{-1/2}$ completes the proof.
\end{proof}
\subsection{The asymptotic orders of magnitude}
Later on we will need estimates for the auxiliary functions involved in the main theorems of Section \ref{sec:introduction}. 
Recall that \eqref{strategy2} was valid for any real $\rho <1$. 
Now let $x$ be a large real keeping in mind that we will choose $x = n$ in the partitions $\partition_{\lambda,\mu}$. 
We will choose $\rho = \rho_{\lambda,\mu}(x)$ such that
\begin{align}
x = \rho (\lambda\Phi'_{\mathbb{P}_2}(\rho)+\mu\Phi'_{\mathbb{P}^2}(\rho)).
\label{eq:def_saddle_point}
\end{align}
From Lemma \ref{maintermstheorem} it follows that for $\lambda>0$ the relationship between $x$ and $\rho$ is indeed well-defined, injective and that $\rho \to 1^-$ as $x \to \infty$.
\begin{proposition} 
\label{propositionmagnitude}
Let $\lambda>0$ and denote by $X=X_{\lambda,\mu} = (\log(\frac{1}{\rho_{\lambda,\mu}}))^{-1}$.
One has as $x\to\infty$ that
\begin{align} \label{eq:magnitudes0}
X &= \bigg(\frac{ \frac{x}{\lambda} \log \frac{x}{\lambda}}{4 \zeta(2)(\log \log \frac{x}{\lambda}  - \log 2 +M )}\bigg)^{\frac{1}{2}}
\bigg(1 + \frac{1}{2} \frac{\log \log \frac{x}{\lambda}}{\log \frac{x}{\lambda}} - \frac{1}{2} \frac{\log \log \log \frac{x}{\lambda}}{\log \frac{x}{\lambda}} + O\bigg(\frac{1}{\log x}\bigg)\bigg).
\end{align}
Furthermore,
\begin{align} \label{magnitudes1}
x \log \frac{1}{\rho(x)} 
&= 
\bigg( \frac{4 \zeta(2) \lambda x ( \log \log \frac{x}{\lambda}- \log 2 +M )}{\log \frac{x}{\lambda}} \bigg)^{\frac{1}{2}}  \bigg(1 - \frac{1}{2} \frac{\log \log \frac{x}{\lambda}}{\log \frac{x}{\lambda}} + \frac{1}{2} \frac{\log \log \log \frac{x}{\lambda}}{\log \frac{x}{\lambda}} + O\bigg(\frac{1}{\log x}\bigg)\bigg)
\end{align}
as $x \to \infty$. For the function $\Phi_{\mathbb{P}_2}$ and its derivatives, one has that for all $m\in\N_0$
\begin{align} \label{magnitudes3}
\lambda\Phi_{\mathbb{P}_2,(m)}(\rho(x)) 
&=
\lambda\bigg(\rho \frac{d}{d\rho}\bigg)^m \Phi_{\mathbb{P}_2}(\rho) \nonumber \\
&= 
\Gamma(m+1)\left(\frac{x}{\lambda}\right)^{\frac{m+1}{2}}\bigg(\frac{\log \frac{x}{\lambda}}{4\zeta(2) (\log \log \frac{x}{\lambda}-\log 2+M  )}\bigg)^{\frac{m-1}{2}} \bigg(1 + O\bigg(\frac{\log \log x}{\log x}\bigg)\bigg),
\end{align}
as $x \to \infty$.
\end{proposition}	

\begin{proof}
We begin with the case $\lambda =1$ and $\mu =0$ and thus have to consider $\Phi_{\mathbb{P}_2}$ only.
Let us suppose that $x$ is sufficiently large in which case $\rho$ will be very close to $1$ and so $X = \frac{1}{\log(\frac{1}{\rho})}$ is also large. We have shown in Lemma \ref{maintermstheorem} that
\begin{align} \label{asympaux01}
x = \rho \frac{d}{d \rho} \Phi_{\mathbb{P}_2}(\rho) = \frac{2\zeta(2)X^2}{\log X}(\log\log X+ M)\bigg(1+ O\bigg(\frac{1}{\log X}\bigg)\bigg).
\end{align}
Taking the logarithm of \eqref{asympaux01} implies that
\begin{align}
\log x &= 2 \log X  - \log \log X + \log \log \log X+ \log(2 \zeta(2)) + \frac{M}{\log \log X} + O\bigg(\frac{1}{(\log \log X)^2}\bigg).
\label{eq:asympaux02}
\end{align}
This implies that $\log x\ll \log X \ll \log x$. 
Furthermore, taking the logarithm of \eqref{eq:asympaux02} gives
\begin{align*}
\log \log x &= \log \log X + \log 2 +O\left(\frac{\log \log X}{\log X}\right), \nonumber \\
\log \log \log x &= \log \log \log X + O\left(\frac{1}{\log \log X }\right).
\end{align*}
We now plug these into \eqref{asympaux01} so that
\begin{align*}
x 
&= 
\frac{2 \zeta(2)X^2 (\log \log x-\log 2+M +O(\frac{\log \log x}{\log x} ) )}{\frac{1}{2}\log x + \frac{1}{2}\log \log x - \frac{1}{2}\log \log \log x -\frac{\log(2 \zeta(2))}{2}+O(\frac{1}{\log \log X})}\bigg(1+ O\bigg(\frac{1}{\log x}\bigg)\bigg)\\
&=
\frac{2 \zeta(2)X^2\left(\log \log x-\log 2+M \right)}{\frac{1}{2}\log x + \frac{1}{2}\log \log x - \frac{1}{2}\log \log \log x -\frac{\log(2 \zeta(2))}{2}}\bigg(1+ O\bigg(\frac{1}{\log x}\bigg)\bigg).
\end{align*}
Solving for $X$ yields
\begin{align} \label{asympaux02}
X &= \bigg(\frac{x \log x}{4 \zeta(2)(\log \log x  - \log 2 +M )}\bigg)^{\frac{1}{2}}
\bigg(1 + \frac{1}{2} \frac{\log \log x}{\log x} - \frac{1}{2} \frac{\log \log \log x}{\log x} + O\bigg(\frac{1}{\log x}\bigg)\bigg).
\end{align}
This completes the proof of \eqref{eq:magnitudes0} for the case $\lambda =1$.
The case $\lambda>0$ and $\mu =0$ follows also immediately.
Indeed, in this case \eqref{eq:def_saddle_point} reduces to $x = \lambda \rho \Phi'_{\mathbb{P}_2}(\rho)$ and we thus have to replace $x$ by $x/\lambda$ in \eqref{asympaux02} only.
It remains to check the chase  $\lambda>0$ and $\mu$ arbitrary.
We thus have to solve
\begin{align} \label{asympaux21}
x = \lambda \rho \frac{d}{d \rho} \Phi_{\mathbb{P}_2}(\rho) + \mu \rho \frac{d}{d \rho} \Phi_{\mathbb{P}^2}(\rho),
\end{align}
We now know from \cite[Lemma~3.1]{gafniprimepowers} that 
\begin{align}
\rho \Phi'_{\mathbb{P}^2}(\rho)
=
\frac{\Gamma(3/2)\zeta(3/2)X^{3/2}}{\log X} \left(1+ O\left(\frac{1}{\log X}\right)\right).
\label{eq:saddle_gafni_derivative}
\end{align}
Thus, combining Theorem~\ref{maintermtheorem1} and \eqref{eq:saddle_gafni_derivative},we can rewrite \eqref{asympaux21} as
\begin{align} \label{asympaux31}
x = \lambda\frac{2\zeta(2)X^2}{\log X}(\log\log X+ M)\bigg(1+ O\bigg(\frac{1}{\log X}\bigg)\bigg).
\end{align}
Now, \eqref{asympaux31} and \eqref{asympaux01} agree and we can thus use essentially the same computation as above.
This completes the proof of \eqref{eq:magnitudes0}.
For the proof of \eqref{magnitudes1}, we have just to combine $x \log \frac{1}{\rho} = xX^{-1}$ and \eqref{eq:magnitudes0}.

It remains to prove \eqref{magnitudes3}. From \eqref{asympaux01} and Lemma \ref{maintermstheorem} we can also write
\begin{align} \label{asympaux03}
\lambda\Phi_{\mathbb{P}_2,(m)}(\rho) =  \lambda\Gamma(m+1)xX^{m-1}\bigg(1+ O\bigg(\frac{1}{\log X}\bigg)\bigg).
\end{align}
Inserting \eqref{asympaux02} into \eqref{asympaux03} yields the results for $m=0$ and for $m \ge 1$ in \eqref{magnitudes3}	.
\end{proof}
%	
%%%%%%%%%%%%%%%%%%%%%%%%%%%%%%%%%%%%%%%%%%%%%%%%%%%%%%%%%%%%%%%%%%%%%%
\section{The minor arcs} \label{sec:minorarcs}
%%%%%%%%%%%%%%%%%%%%%%%%%%%%%%%%%%%%%%%%%%%%%%%%%%%%%%%%%%%%%%%%%%%%%%
\subsection{The fundamental estimate for the double Weyl sum}
In order to show that every sufficiently large odd number is a sum of three primes in \cite{vinogradov}, Vinogradov studied the Weyl sum $S_1(\beta,x) := \sum_{p \le x} \e(\beta x)$ where the sum is taken over primes $p$ and $\beta \in \R$. Vinogradov showed that if $|\beta - \frac{a}{q}| \le \frac{1}{q^2}$ with $(a,q)=1$, then one has that $S_1(\beta,x)$ is bounded by $(\tfrac{x}{\sqrt{q}} + x^{\frac{4}{5}} + \sqrt{x}\sqrt{q}) (\log 2x)^3$. In \cite[$\mathsection$25]{davenport}, Davenport states that this bound is sharp, even in the special case $\beta = a/q$, but that if the hypothesis is weakened then the corresponding bound will also be weakened. For our purposes, we need to extend the hypothesis by creating a generalization which is weaker but whose special case contains Vinogradov's bound. Our result, interesting in its own right and useful for other purposes outside the scope of the paper, is as follows.
\begin{proposition} \label{extendedvinogradovhypothesis}
Let $\beta \in \R$. If $a \in \Z$ and $q \in \N$, $\Upsilon>0$ and
\begin{align*} 
\bigg|\beta - \frac{a}{q}\bigg| \le \frac{\Upsilon}{q^2} \quad \textnormal{with} \quad (a,q)=1,
\end{align*}
then one has that
\begin{align*}
S_1(\beta,x) = \sum_{p \le x} \e(\beta p) \ll \Upsilon \bigg(\frac{x}{\sqrt{q}} + x^{\frac{4}{5}} + \sqrt{x}\sqrt{q}\bigg)(\log x)^3, 
\end{align*}
where the sum runs over primes $p$.
\end{proposition}
\begin{proof}
The technique to prove this relies on Vaughan's identity, see \cite{vaughanidentity} and \cite[$\mathsection$25]{davenport} as well as \cite{vaughanbook}. We define the truncated the Dirichlet series 
\begin{align*}
F(s) = \sum_{m \le U} \Lambda(m) m^{-s}, \quad G(s) = \sum_{d \le V} \mu(d) d^{-s}
\end{align*}
and note the identity
\begin{align} \label{vaughanaux1}
-\frac{\zeta'}{\zeta}(s) = F(s)-\zeta(s) F(s) G(s) - \zeta'(s) G(s) + \bigg(-\frac{\zeta'}{\zeta}(s)-F(s)\bigg)(1-\zeta(s)G(s)),
\end{align}
valid for $\sigma > 1$. The Dirichlet coefficients of the four functions on the right-hand side of \eqref{vaughanaux1} can be calculated and we see that
\begin{align*}
\Lambda(n) = a_1(n) + a_2(n) + a_3(n) + a_4(n),
\end{align*}
where the first two coefficients are given by 
\begin{align*}
a_1(n) = 
\begin{cases}
\Lambda(n) \quad &\mbox{if $n \le U$}, \nonumber \\
0 \quad &\mbox{if $n > U$},
\end{cases}
\quad
a_2(n) = - \sum_{\substack{mdr = n \\ m \le U \\ d \le V}} \Lambda(m) \mu(d),
\end{align*}
and the last two coefficients are
\begin{align*}
a_3(n) = \sum_{\substack{hd = n \\ d \le V}} \mu(d) \log h, \quad \textnormal{and} \quad a_4(n) = - \sum_{\substack{mk = n \\ m > U \\ k>1}} \Lambda(m) \bigg(\sum_{\substack{d | k \\ d \le V}} \mu(d)\bigg).
\end{align*}
From this we can construct the sum
\begin{align*}
\sum_{n \le N} f(n) \Lambda(n) = S_1 + S_2 + S_3 + S_4 \quad \textnormal{with} \quad S_i = \sum_{n \le N}f(n) a_i(n) \quad \textnormal{for} \quad i=1,2,3,4.
\end{align*}
This can be shown to satisfy
\begin{align*}
\sum_{n \le N} f(n) \Lambda(n) &\ll U + (\log N) \sum_{t \le UV} \max_{w} \bigg|\sum_{w \le r \le N/t} f(rt)\bigg| \nonumber \\
&\quad +N^{\frac{1}{2}}(\log N)^3 \max_{U \le M \le N/V} \max_{V \le j \le N/M} \bigg(\sum_{v < k \le N/M} \bigg|\sum_{\substack{M < n \le 2M \\ m \le N/k \\ m \le N/j}} f(mj) \overline{f(mk)}\bigg|\bigg)^{\frac{1}{2}}.
\end{align*}
Only the last term represents a difficulty from the known case in which $\Upsilon = 1$. Suppose that $a,q$ and $\Upsilon$ are as in \eqref{extendedvinogradovhypothesis}. From this we can write
\begin{align} \label{auxvinogradov0}
\gamma = \beta - a/q \implies \beta = \gamma + a/q \implies |\gamma| \le \frac{\Upsilon}{q^2}.
\end{align}
Recall that for positive integers $N_1$ and $N_2$ one has
\begin{align*}
\sum_{n=N_1}^{N_2} \e(n \beta) = \frac{\e((N_2+1)\beta)-\e(N_1 \beta)}{\e(\beta)-1} \ll \min \bigg(N_2 - N_1, \frac{1}{||\beta||}\bigg),
\end{align*}
where $||x||$ denotes the distance from $x$ to the nearest integer. Hence
\begin{align} \label{auxvinogradov1}
\sum_{t \le T} \max_w \bigg|\sum_{w \le r \le N/t} \e(r t \beta)\bigg| \ll \sum_{t \le T} \min\bigg(\frac{N}{t}, \frac{1}{||t \beta||}\bigg).
\end{align}
For now we assume that the right-hand side of \eqref{auxvinogradov1} is
\begin{align} \label{tempbound}
\ll \Upsilon \bigg(\frac{N}{q}+T+q\bigg)\log(2qT),
\end{align}
for $\beta$ satisfying \eqref{auxvinogradov0}, a bound that we will prove shortly. In that case, following \cite[$\mathsection$25]{davenport} we can deduce that $S_1(\beta)$ satisfies
\begin{align*}
S_1(\beta) &\ll U + \Upsilon\bigg(\frac{N}{q}+UV+q\bigg)(\log 2qN)^2 \nonumber \\
&\quad  + N^{\frac{1}{2}}(\log N)^3 \max_{U \le M \le N/V} \max_{V < j \le N/M} \bigg(\sum_{V< k \le N/M} \min \bigg(M, \frac{\Upsilon}{||(k-j)\beta||} \bigg)\bigg)^{\frac{1}{2}}.
\end{align*}
The last term of the above equation is
\begin{align*}
\ll N^{\frac{1}{2}}(\log N)^3 \max_{U \le M \le N/V} \bigg(M + \sum_{1< m \le N/M} \min \bigg(\frac{N}{m}, \frac{\Upsilon}{||m\beta||} \bigg)\bigg)^{\frac{1}{2}}.
\end{align*}
By employing \eqref{tempbound} this expression is seen to be 
\begin{align*}
&\ll N^{\frac{1}{2}}(\log N)^3 \max_{U \le M \le N/V}  \bigg(M + \Upsilon \bigg(\frac{N}{M}+\frac{N}{q}+q\bigg)\bigg)^{\frac{1}{2}} (\log qN)^{\frac{1}{2}} \nonumber \\
&\ll \Upsilon (NV^{-1/2} + NU^{-1/2} + Nq^{-1/2} + N^{\frac{1}{2}}q^{\frac{1}{2}}) (\log q N)^3.
\end{align*}
Hence, collecting all the terms we arrive at
\begin{align*}
S_1(\beta) \ll \Upsilon (UV + q + NU^{-1/2}+NV^{-1/2}+Nq^{-1/2}+N^{\frac{1}{2}}q^{\frac{1}{2}}) (\log q N)^3.
\end{align*}
If $q > N$ then the lemma holds trivially. We then assume that $q \le N$ and obtain the statement of the lemma by taking $U = V = N^{\frac{1}{2}}$.

The last item in the proof is to show that indeed \eqref{tempbound} holds. Take $t=hq+r$ with $1 \le r \le q$ and put $\gamma = \beta - a/q$. In that case
\begin{align*}
\sum_{t \le T} \min \bigg(\frac{N}{t}, \frac{\Upsilon}{||t \beta ||}\bigg) \ll \sum_{0 \le h \le T/q} \sum_{r=1}^q \min \bigg(\frac{N}{hq+r}, \Upsilon \bigg|\bigg| ra/q + hq \gamma + r \gamma \bigg|\bigg|^{-1} \bigg).
\end{align*}
First we consider $h=0$ and $1 \le r \le \tfrac{q}{2\Upsilon}$. For such $r$ we have $|r \gamma| \le \frac{1}{2q}$, so that the contribution from these terms is 
\begin{align*}
\ll \sum_{1 \le r \le \frac{q}{2 \Upsilon}} \frac{1}{||\tfrac{ra}{q}||-\tfrac{1}{2q}} \ll q \log q.
\end{align*}
We have at most $O(\Upsilon)$ such sums with $r \in [1,q]$, and hence the contribution of all $r$ in the interval $[1,q]$ is $O(\Upsilon q \log q)$.
For all remaining terms we have $hq + r \gg (h+1)q$. Let $h \in \Z$ be given. Then for any $b \in \Z_{>0}$ the interval $J=[b/q,(b+1)/q]$ has length $1/q$. Let $J_1$ and $J_2$ be subintervals of $[b/q,(b+1+\Upsilon)/q]$ such that their lengths $|J_1|, |J_2|$ satisfy $|J_1|, |J_2| \le \Upsilon /q$ and such that $\tfrac{ra}{q} \in J_1 \cup J \cup J_2$. There are at most $2\Upsilon +2$ values of $r$ with $1 \le r \le q$ for which
\begin{align*}
\frac{ra}{q} + h q \gamma + r \gamma \in [b/q, (b+1)q] \; \modu 1 \implies \frac{ra}{q} + r \gamma \in [b/q - h \gamma q, (b+1)/q - h \gamma q] \; \modu 1. \nonumber
\end{align*}
Therefore, we end up with
\begin{align*}
\sum_{0 \le h \le T} \sum_{r=1}^q \min \bigg(\frac{N}{(h+1)q}, \Upsilon \bigg|\bigg|\frac{ra}{q} + hq\gamma + r\gamma\bigg|\bigg|^{-1} \bigg) &\ll \sum_{0 \le h \le T/q} \bigg(\frac{N}{(h+1)q} + \Upsilon q \log 2q \bigg) \nonumber \\
&\ll \Upsilon \bigg(\frac{N}{q}+T+q\bigg) \log 2qT,
\end{align*}
thereby completing the final ingredient of the proposition.
\end{proof}
Before proceeding to the estimate on the Weyl sum over two primes, namely $\sum_{p_1 p_2 \le x} \e(\alpha p_1 p_2)$, we need an auxiliary technical lemma. 
%\begin{color}{red}
\begin{lemma} 
\label{lem:min_max_for_minor}
Let $F$, $G_1$, $G_2$ and $G_3$ be continuous, real valued functions on $(0,\infty)$ such that F is strictly decreasing and all $G_i$ are increasing.
Further, suppose that for $i=1,2,3$
\begin{align}
\lim_{x\to\infty} F(x) = \lim_{x\to0 } G_i(x) = 0
\ \text{ and } \
\lim_{x\to0} F(x) = \lim_{x\to\infty } G_i(x) = \infty.
\end{align} 
Set $G(x):= \max\{G_1(x),G_2(x),G_3(x)\}$ and $H(x):=\max\{F(x),G(x)\}$. 
We then have
\begin{align}
\min_{x\in(0,\infty)} H(x) 
=
F(\min\{M_1,M_2,M_3\})
=
G(\min\{M_1,M_2,M_3\}),
\end{align}
where $M_i$ is the solution of the equation $F(M_i) = G_i(M_i)$ for $i=1,2,3$.
\end{lemma}
The assumptions on $F$, $G_1$, $G_2$ and $G_3$ imply that the solutions $M_i$ exist and are unique.
Furthermore, the lemma also holds for more than three functions $G_i$ or just two of them.
We have formulated Lemma~\ref{lem:min_max_for_minor} with three $G_i$ since we need it in this form below.

\begin{proof}[Proof of Lemma~\textnormal{\ref{lem:min_max_for_minor}}]
We can assume that $M_1\leq M_2\leq M_3$. Otherwise we just relabel the functions $G_i$.
First, we show that 
\begin{align}
G_1(M_{1})\geq G_2(M_{2})\geq G_2(M_{1}).
\label{eq:inequal_G_i1}
\end{align}
Since $G_2$ is increasing and $M_2\geq M_1$, the second inequality follows immediately.
To show the first inequality, we use the definition of $M_i$ and that $F$ is decreasing. 
We have
\begin{align*}
%	G_2(M_{2})=F(M_{2})\leq F(M_{1})=G_1(M_{1}).
G_1(M_{1})=F(M_{1})\geq F(M_{2})=G_2(M_{2}).
\end{align*}
Thus \eqref{eq:inequal_G_i1} holds and therefore $G_1(M_{1})\geq G_2(M_{1})$.
%Since $G_2$ is increasing and $M_1 \leq M_2$, we get $G_1(M_{1})\geq G_2(M_{1})$.
Similarly $G_1(M_{1})\geq G_3(M_{1})$. Thus
\begin{align}
F(M_1)=G_1(M_1)=\max\{G_1(M_1), G_2(M_1), G_3(M_1)\} =G(M_1)= H(M_1). 
\end{align}
Moreover, since $F$ is decreasing and $G_1$ is increasing, we get
\begin{align}
H(M_1) \leq 
\begin{cases}
F(x) & \text{for }x\leq M_1,\\
G_1(x)& \text{for }x\geq M_1.
\end{cases}
\end{align}
This completes the proof.
\end{proof}

The last result we shall need is a bilinear form for exponential sums from \cite[$\mathsection$13]{iwanieckowalski}.
\begin{lemma} \label{lemmaiwanieckowalski}
Let  $\alpha \in \R$ and $a \in \Z$ as well as $q \in \N$ such that 
\begin{align}
\bigg|\alpha-\frac{a}{q}\bigg| \le \frac{1}{q^{2}} \quad \textnormal{with} \quad (a,q)=1. \nonumber
\end{align}
For any complex numbers $\xi_m, \eta_n$ with $|\xi_m| \le 1$ and $|\eta_n| \le 1$ we have
\begin{align*}
\sumtwo_{\substack{mn \le x \\ m>M \\ n>N}} \xi_m \eta_n \e(\alpha m n) \ll \bigg(\frac{x}{M}+\frac{x}{N}+\frac{x}{q}+q\bigg)^{\frac{1}{2}}x^{\frac{1}{2}}(\log x)^2.%, \textnormal{ with $(a,q)=1 $ and $|\alpha-a/q| \le q^{-2}$.} 
\end{align*}
\end{lemma}
\begin{proof}
See Lemma 13.8 from \cite{iwanieckowalski}.
\end{proof}

We are now in a position to prove the fundamental estimate needed to bound the minor arcs. 
\begin{theorem} \label{doublevinogradov}
Let $\alpha \in \R$. If $a \in \Z$ and $q \in \N$ are such that
\begin{align*}
\bigg|\alpha - \frac{a}{q}\bigg| \le \frac{1}{q^2} \quad \textnormal{with} \quad (a,q)=1,
\end{align*}
then one has that
\begin{align*}
S_2(\alpha, X) := \sum_{p_1 p_2 \le X} \e (\alpha p_1 p_2) \ll \frac{X}{q^{\frac{1}{6}}} (\log X)^{\frac{7}{3}} + X^{\frac{16}{17}} (\log X)^{\frac{39}{17}} + X^{\frac{7}{8}}q^{\frac{1}{8}} (\log X)^{\frac{9}{4}},
\end{align*}
where the sum is taken over primes $p_1$ and $p_2$.
\end{theorem}
\begin{proof}
The idea is to break the summation $S_2(\alpha, X)$ into different types of sums. From Figure \eqref{hyperbola1} we can see that we want to bound the area underneath the hyperbola $uv \le X$. 
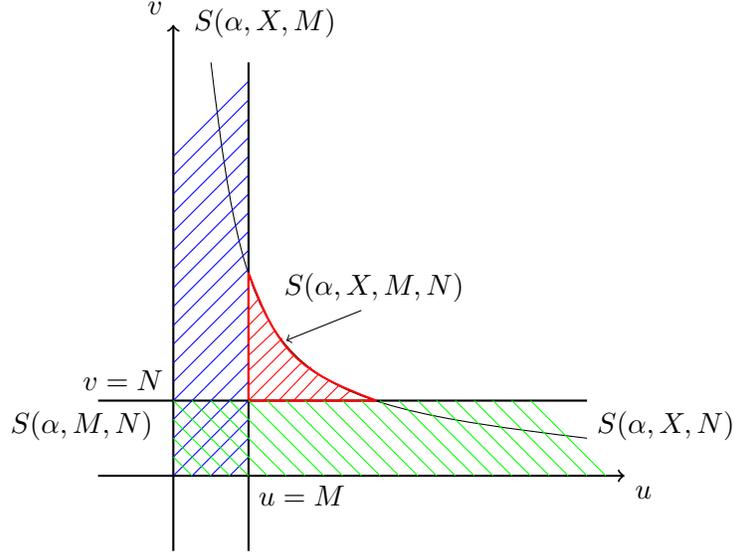
\begin{figure}[h]
\centering
\begin{tikzpicture}
\draw[thick,->] (-1,0) -- (6,0) node[anchor=north west] {$u$};
\draw[thick,->] (0,-1) -- (0,6) node[anchor=south east] {$v$};
\draw (.5,5.5) .. controls (1,1.3) and (1.3,1) .. (5.5,.5) ;
\draw[thick] (-1,1) -- (5.5,1);
\draw[thick] (1,-1) -- (1,5.5);
\draw (1,0) node[anchor=north west] {$u=M$};
\draw (0,1) node[anchor=south east] {$v=N$};

\draw[thick, red] (1,1) -- (1,2.7);
\draw[thick, red] (1,1) -- (2.7,1);
\draw[thick, red] (1,2.7) .. controls (1.39,1.65) and (1.65,1.39) .. (2.7,1);
\draw[red] (1,1) -- (1.6,1.6);
\draw[red] (1.2,1) -- (1.7,1.5);
\draw[red] (1.4,1) -- (1.83,1.45);
\draw[red] (1.6,1) -- (1.93,1.35);
\draw[red] (1.8,1) -- (2.07,1.27);
\draw[red] (2,1) -- (2.2,1.2);
\draw[red] (2.2,1) -- (2.35,1.13);
\draw[red] (1,1.2) -- (1.5,1.7);
\draw[red] (1,1.4) -- (1.45,1.83);
\draw[red] (1,1.6) -- (1.35,1.93);
\draw[red] (1,1.8) -- (1.27,2.07);
\draw[red] (1,2) -- (1.2,2.2);
\draw[red] (1,2.2) -- (1.13,2.35);

\draw[blue] (0.75,0) -- (1,0.25);
\draw[blue] (0.5,0) -- (1,0.5);
\draw[blue] (0.25,0) -- (1,0.75);
\draw[blue] (0,0) -- (1,1);
\draw[blue] (0,0.25) -- (1,1.25);
\draw[blue] (0,.5) -- (1,1.5);
\draw[blue] (0,.75) -- (1,1.75);
\draw[blue] (0,1) -- (1,2);
\draw[blue] (0,1.25) -- (1,2.25);
\draw[blue] (0,1.5) -- (1,2.5);
\draw[blue] (0,1.75) -- (1,2.75);
\draw[blue] (0,2) -- (1,3);
\draw[blue] (0,2.25) -- (1,3.25);
\draw[blue] (0,2.5) -- (1,3.5);
\draw[blue] (0,2.75) -- (1,3.75);
\draw[blue] (0,3) -- (1,4);
\draw[blue] (0,3.25) -- (1,4.25);
\draw[blue] (0,3.5) -- (1,4.5);
\draw[blue] (0,3.75) -- (1,4.75);
\draw[blue] (0,4) -- (1,5);
\draw[blue] (0,4.25) -- (1,5.25);

\draw[green] (.25,0) -- (0,.25);
\draw[green] (.5,0) -- (0,.5);
\draw[green] (.75,0) -- (0,.75);
\draw[green] (1,0) -- (0,1);
\draw[green] (1.25,0) -- (.25,1);
\draw[green] (1.5,0) -- (0.5,1);
\draw[green] (1.75,0) -- (.75,1);
\draw[green] (2,0) -- (1,1);
\draw[green] (2.25,0) -- (1.25,1);
\draw[green] (2.5,0) -- (1.5,1);
\draw[green] (2.75,0) -- (1.75,1);
\draw[green] (3,0) -- (2,1);
\draw[green] (3.25,0) -- (2.25,1);
\draw[green] (3.5,0) -- (2.5,1);
\draw[green] (3.75,0) -- (2.75,1);
\draw[green] (4,0) -- (3,1);
\draw[green] (4.25,0) -- (3.25,1);
\draw[green] (4.5,0) -- (3.5,1);
\draw[green] (4.75,0) -- (3.75,1);
\draw[green] (5,0) -- (4,1);
\draw[green] (5.25,0) -- (4.25,1);
\draw[green] (5.5,0) -- (4.5,1);
\draw[green] (5.75,0) -- (4.75,1);
\draw (5.5,1) node[anchor=north west] {$S(\alpha,X,N)$};
\draw (2.3,5.7) node[anchor=south east] {$S(\alpha,X,M)$};
\draw (4,2.5) node[anchor=east] {$S(\alpha,X,M,N)$};
\draw (-2.3,1) node[anchor=north west] {$S(\alpha,M,N)$};
\draw [thin, ->] (2.5,2.2) -- (1.5,1.8);
\end{tikzpicture}
\caption{Hyperbola representation of $uv \le X$.} \label{hyperbola1}
\end{figure}
We set $M=N$. Therefore, we split $S_2(\alpha,X)$ into these four types of sums. First we compute the red area and then the two green and blue rectangles. Due to the double computation of the small square we subtract that common area from the total contribution. This leads us to
\begin{align*}
S_2(\alpha, X) &= \sumtwo_{\substack{p_1 p_2 \le X \\ p_1 > M \\ p_2 > M}} \e (\alpha p_1 p_2) + \sumtwo_{\substack{p_1 p_2 \le X \\ p_1 \le M}} \e (\alpha p_1 p_2) + \sumtwo_{\substack{p_1 p_2 \le X \\ p_2 \le M}} \e (\alpha p_1 p_2) - \sumtwo_{\substack{p_1 \le M \\ p_2 \le M}} \e (\alpha p_1 p_2) \nonumber \\
&=: \Sigma_1 + \Sigma_2 + \Sigma_3 - \Sigma_4.
\end{align*}
Next we define $\xi_m$ and $\eta_n$ as the indicator functions
\begin{align*}
\xi_m := \begin{cases}
1, \quad &\mbox{if $m=p_1$ is a prime}, \nonumber \\
0, \quad &\mbox{otherwise},
\end{cases}
\quad
\eta_n := \begin{cases}
1, \quad &\mbox{if $n=p_2$ is a prime}, \nonumber \\
0, \quad &\mbox{otherwise}.
\end{cases}
\end{align*}
By using Lemma \eqref{lemmaiwanieckowalski} we can bound the first sum as 
\begin{align} \label{vaughanfirstsum}
\Sigma_1 &= \sumtwo_{\substack{m_1 m_2 \le X \\ m_1 > M \\ m_2 > M}} \xi_m \eta_n \e (\alpha m_1 m_2) \ll \bigg(\frac{X}{M}+\frac{X}{q}+q\bigg)^{\frac{1}{2}}X^{\frac{1}{2}}(\log X)^2 \nonumber \\
&\ll X(\log X)^2 \bigg(\frac{1}{M^{\frac{1}{2}}} + \frac{1}{\sqrt{q}} + \frac{\sqrt{q}}{\sqrt{X}} \bigg). %\ll \frac{X(\log X)^2}{M^{\frac{1}{2}}}
\end{align}
Note that $\Sigma_2$ and $\Sigma_3$ are symmetric and thus we may use the first one as a representative. We then have that $\Sigma_2$ is given by
\begin{align} \label{representativesum}
\Sigma_2 = \sumtwo_{\substack{p_1 p_2 \le X \\ p_1 \le M}} \e(\alpha p_1 p_2) = \sum_{p_1 \le M} \sum_{p_2 \le X/p_1} \e(\alpha p_1 p_2).
\end{align}
Let $\beta = p_1 \alpha$. We use Dirichlet's theorem to choose $a' \in \Z$ and $q' \in \N$ such that
\begin{align*}
\bigg|\beta - \frac{a'}{q'}\bigg| \le \frac{1}{(q')^2} \quad \textnormal{with} \quad (a',q')=1.
\end{align*}
Then picking up from \eqref{representativesum} we arrive at 
\begin{align} \label{representativesumabs}
\bigg|\sum_{p_1 \le M} \sum_{p_2 \le X/p_1} \e(\alpha p_1 p_2)\bigg| \le \sum_{p_1 \le M} \bigg| \sum_{p_2 \le X/p_1} \e(\alpha p_1 p_2) \bigg| = \sum_{p_1 \le M} \bigg| \sum_{p_2 \le X/p_1} \e(\beta p_2) \bigg|.
\end{align}
Now we define the following parameter
\begin{align*} 
\Upsilon := \begin{cases}
1, \quad &\mbox{if $q \equiv 0 \; \modu p_1$}, \nonumber \\
p_1, \quad &\mbox{if $q \not\equiv 0 \; \modu p_1$}.
\end{cases}
\end{align*}
To compute the upper bound of $\Sigma_2$ we have to consider two cases: whether $p_1$ divides $q$, in which case $\Upsilon = 1$, or not, in which case $\Upsilon = p_1$. We start with the first case. Since $(a,q)=1$, if $q \equiv 0 \; \modu p_1$ then $(ap_1,q)=(p_1,q)=p_1$. Let $q'=\tfrac{q}{p_1}$, then we have
\begin{align} \label{auvinogradoUpsilon1}
\bigg|p_1 \alpha - \frac{p_1a}{q}\bigg| \le \frac{p_1^2}{q^2} &\implies \bigg|\alpha - \frac{a}{q}\bigg| \le \frac{p_1}{q^2}, \quad q \equiv 0 \; \modu p_1, \nonumber \\
&\implies \bigg|\alpha - \frac{a}{q}\bigg| \le \frac{\Upsilon}{q^2} = \frac{1}{q^2}, \quad (a,q)=1.
\end{align}
Relation \eqref{auvinogradoUpsilon1} along with \eqref{representativesum} and \eqref{representativesumabs} and Proposition \eqref{extendedvinogradovhypothesis} with $\Upsilon = 1$ yields
\begin{align} \label{vaughansfirstdivision}
\mathfrak{s}_1 &
:=
\bigg|\sumtwo_{\substack{p_1 p_2 \le X \\ p_1 \le M \\ q \equiv 0 \; \modu p_1}} \e(\alpha p_1 p_2)\bigg| \nonumber \\
&\ll 
\sum_{\substack{p_1 \le M \\ q \equiv 0 \; \modu p_1}} \bigg(\frac{X}{p_1 \sqrt{q}} + \bigg(\frac{X}{p_1}\bigg)^{\frac{4}{5}} + \sqrt{q}\sqrt{X}{\sqrt{p_1}} \bigg) \bigg(\log \frac{X}{p_1}\bigg)^3 \nonumber \\
&\ll 
\sum_{\substack{p_1 \le M \\ q \equiv 0 \; \modu p_1}} \frac{X}{p_1 \sqrt{q}}\bigg(\log \frac{X}{p_1}\bigg)^3 
+ 
\sum_{\substack{p_1 \le M \\ q \equiv 0 \; \modu p_1}} \bigg(\frac{X}{p_1}\bigg)^{\frac{4}{5}}\bigg(\log \frac{X}{p_1}\bigg)^3  + \sum_{\substack{p_1 \le M \\ q \equiv 0 \; \modu p_1}} \frac{\sqrt{X}\sqrt{q}}{\sqrt{p_1}} \bigg(\log \frac{X}{p_1}\bigg)^3 \nonumber \\
&\ll 
\frac{X(\log X)^3 \log \log M}{\sqrt{q}} + X^{\frac{4}{5}}(\log X)^3 M^{\frac{1}{5}} + \sqrt{X}\sqrt{q}\sqrt{M}(\log X)^3.
\end{align}
For the second case we use $q \not\equiv 0 \; \modu p_1$ so that $\Upsilon = p_1$. Employing a similar argument as in the previous case with $(ap_1,q)=(p_1,q)=1$ we write
\begin{align} \label{auvinogradoUpsilon2}
\bigg|p_1 \alpha - \frac{p_1a}{q}\bigg| \le \frac{p_1^2}{q^2} &\implies \bigg|\alpha - \frac{a}{q}\bigg| \le \frac{p_1}{q^2}, \quad q \not\equiv 0 \; \modu p_1, \nonumber \\
&\implies \bigg|\alpha - \frac{a}{q}\bigg| \le \frac{\Upsilon}{q^2}, \quad (a,q)=1.
\end{align}
Therefore setting $\Upsilon = p_1$ and using Proposition \eqref{extendedvinogradovhypothesis} again along with \eqref{auvinogradoUpsilon2} we see that 
\begin{align} \label{vaughanseconddivision}
\mathfrak{s}_2 &:= \bigg|\sumtwo_{\substack{p_1 \le M \\ p_2 \le X/p_1 \\ q \not\equiv 0 \; \modu p_1}} \e(\alpha p_1 p_2)\bigg| \le \sum_{\substack{p_1 \le M \\ q \not\equiv 0 \; \modu p_1}} \bigg|\sum_{p_2 \le X/p_1} \e(\beta p_2)\bigg| \nonumber \\
&\ll \sum_{\substack{p_1 \le M \\ q \not\equiv 0 \; \modu q}} \Upsilon \bigg(\frac{X}{p_1 \sqrt{q}} + \bigg(\frac{X}{p_1}\bigg)^{\frac{4}{5}} + \frac{\sqrt{q}\sqrt{X}}{\sqrt{p_1}} \bigg) \bigg(\log \frac{X}{p_1}\bigg)^3 \nonumber \\
&\ll \sum_{\substack{p_1 \le M \\ q \not\equiv 0 \; \modu p_1}} \frac{X}{\sqrt{q}}\bigg(\log \frac{X}{p_1}\bigg)^3 + \sum_{\substack{p_1 \le M \\ q \not\equiv 0 \; \modu p_1}} X^{\frac{4}{5}}p_1^{\frac{1}{5}}(\log X)^3 + \sum_{\substack{p_1 \le M \\ q \not\equiv 0 \; \modu p_1}} \sqrt{X}\sqrt{q}\sqrt{p_1}(\log X)^3 \nonumber \\
&\ll \frac{XM(\log X)^3}{\sqrt{q}} + X^{\frac{4}{5}} (\log X)^3 M^{\frac{6}{5}} + \sqrt{X}\sqrt{q}M^{\frac{3}{2}}(\log X)^3. 
\end{align}
Next we note the each term on the far right side of the bound of \eqref{vaughansfirstdivision} is dominated by a corresponding term on the far right side of the bound of \eqref{vaughanseconddivision}. Therefore,
\begin{align} \label{vaughanmerge}
\Sigma_2 &\ll \frac{X(\log X)^3\log\log M}{\sqrt{q}}+X^{\frac{4}{5}}(\log X)^3\log\log M+\sqrt{X}\sqrt{q}\sqrt{M}(\log X)^3 \nonumber\\
& \quad + \frac{XM(\log X)^3}{\sqrt{q}}+X^{\frac{4}{5}}(\log X)^3M^{\frac{6}{5}}+\sqrt{X}\sqrt{q}{M^{\frac{3}{2}}}(\log X)^3
\nonumber\\
&\ll  \max \left ( \frac{XM(\log X)^3}{\sqrt{q}} , X^{\frac{4}{5}}(\log X)^3M^{\frac{6}{5}} , \sqrt{X}\sqrt{q}{M^{\frac{3}{2}}}(\log X)^3 \right).
\end{align}
Next we bound $\Sigma_4$ trivially using the prime number theorem
\begin{align}\label{s3}
\Sigma_4= \bigg|\sum_{p_1\leq M}\sum_{p_2\leq M} \e(\alpha p_1p_2)\bigg| \le \sum_{p_1 \le M} \sum_{p_2 \le M} 1 \ll \bigg(\frac{M}{\log M}\bigg)^2.
\end{align}
The term on the right-hand side of \eqref{s3} is dominated by the last term on the right-hand side of \eqref{vaughanseconddivision}. Now combining \eqref{vaughanfirstsum} and \eqref{vaughanmerge}, we deduce that
\begin{align*}
\Sigma_1+\Sigma_2 + \Sigma_3 - \Sigma_4 &\ll \max \left \{ \frac{X(\log X)^2}{M^{\frac{1}{2}}} ,
\frac{XM(\log X)^3}{\sqrt{q}} , X^{\frac{4}{5}}(\log X)^3M^{\frac{6}{5}} , \sqrt{X}\sqrt{q}{M^{\frac{3}{2}}}(\log X)^3 \right\}. 
\end{align*}
Therefore we arrive at
\begin{align*}
S_2(\alpha, X) 
&\ll  
\inf_{M} \max  \left\{\frac{X(\log X)^2}{M^{\frac{1}{2}}} 
, \frac{XM(\log X)^3}{\sqrt{q}} , X^{\frac{4}{5}}(\log X)^3M^{\frac{6}{5}} , \sqrt{X}\sqrt{q}{M^{\frac{3}{2}}}(\log X)^3\right\}\\
&=\inf_{M} \max  \left\{F(M), G_1(M), G_2(M), G_3(M)\right\}.
\end{align*}
%            \begin{color}{red}
We now use  Lemma~\ref{lem:min_max_for_minor} to minimise this expression.
The assumptions of Lemma~\ref{lem:min_max_for_minor} are clearly fulfilled and thus we are led to
\begin{align*}
S_2(\alpha, X) 
&\ll  F(\min\{M_1,M_2,M_3\})\leq F(M_1)+F(M_2)+F(M_3),    
\end{align*} 
where $M_i$ is the solution of the equation $F(M)=G_i(M)$ for $i=1,2,3$.
Solving these equations, we get the following three values
\begin{align*}
M_1=\frac{q^{\frac{1}{3}}}{(\log X)^{\frac{2}{3}}}, \quad 
M_2=\frac{X^{\frac{2}{17}}}{(\log X)^{\frac{10}{17}}} \quad \textnormal{and} \quad M_3=\frac{X^{\frac{1}{4}}}{q^{\frac{1}{4}}(\log X)^{\frac{1}{2}}}.
\end{align*}
%           \end{color}
%            \begin{color}{blue}
%			Note that by Lemma \ref{lem:min_max_for_minor}, we can say that at $M_0$ two of the four terms on the right-hand side above must be equal. Since $F(M)$ is a decreasing function and $G(M)$ is an increasing function we only need to solve for the pairs involving $F(M)$. Solving for $M$ the equations $F(M)=G_i(M)$ with $i=1,2,3$ we get the following three values
%			\begin{align*}
%				M_1=\frac{q^{\frac{1}{3}}}{(\log X)^{\frac{2}{3}}}, \quad M_2=\frac{X^{\frac{2}{17}}}{(\log X)^{\frac{10}{17}}} \quad \textnormal{and} \quad M_3=\frac{X^{\frac{1}{4}}}{q^{\frac{1}{4}}(\log X)^{\frac{1}{2}}}.
%			\end{align*}
%			Putting the values in Lemma \ref{lem:min_max_for_minor} we conclude that
%			\begin{align*}
%				S_2(\alpha,X)\ll \inf H(M)=\max\{F(M_1),F(M_2),F(M_3)\}\leq F(M_1)+F(M_2)+F(M_3).    
%			\end{align*}
%            \end{color}
The values of $F$ at $M_i$ for $i=1,2,3$ are
\begin{align*}
F(M_1)= \frac{X(\log X)^{\frac{7}{3}}}{q^{\frac{1}{6}}}, \quad
F(M_2) = X^{\frac{16}{17}}(\log X)^{\frac{39}{17}} \quad \textnormal{and} \quad
F(M_3) =X^{\frac{7}{8}}q^{\frac{1}{8}}(\log X)^{\frac{9}{4}},
\end{align*}
and hence the above equation becomes
\begin{align*}
S_2(\alpha, X) &\ll \frac{X(\log X)^2}{q^{\frac{1}{6}}(\log X)^{-\frac{1}{3}}}  +\frac{X(\log X)^2}{X^{\frac{1}{17}}(\log X)^{-\frac{5}{17}}}    +\frac{X(\log X)^2}{X^{\frac{1}{8}}q^{-\frac{1}{8}}(\log X)^{-\frac{1}{4}}}\\
&= \frac{X (\log X)^{\frac{7}{3}}}{q^{\frac{1}{6}}}  + 
X^{\frac{16}{17}}(\log X)^{\frac{39}{17}} +
X^{\frac{7}{8}}q^{\frac{1}{8}}(\log X)^{\frac{9}{4}},
\end{align*}
and the proof of the theorem is now completed.
\end{proof}

\subsection{Estimate of $\Phi_{\mathbb{P}_2}(\rho \e (\alpha))$ on $\mathfrak{m}$}
From the definition of $\Phi$ we can write
\begin{align*}
\Phi_{\mathbb{P}_2}(\rho \e (\alpha)) = \sum_{j=1}^\infty \frac{1}{j} \sum_{p_1} \sum_{p_2} e^{-p_1 p_2 j /X} \e(j p_1 p_2 \alpha).
\end{align*}
The idea is now to employ the identity
\begin{align*}
e^{-p_1 p_2 j /X} = \int_{p_1 p_2}^\infty jX^{-1}e^{-yj/X}dy
\end{align*}
in order to write the two infinite sums over primes as the following combined sum
\begin{align*}
\sum_{p_1} \sum_{p_2} e^{-p_1 p_2 j /X} \e(j p_1 p_2 \alpha) =\int_{2 \times 2}^\infty jX^{-1}e^{-yj/X} \sum_{p_1 p_2 \le y} \e(jp_1 p_2 \alpha) dy.
\end{align*}
We shall make use of the coarse but useful bound
\begin{align*}
\int_{4}^\infty jX^{-1}e^{-yj/X} \sum_{p_1 p_2 \le y} \e(jp_1 p_2 \alpha) dy\ll \int_0^\infty y j X^{-1}e^{-yk/X}dy.
\end{align*}
Upon integrating by parts we see that for any $\lambda > 0$ we have
\begin{align} \label{ibpminor}
\int_{4}^\infty y^\lambda j X^{-1} e^{-y j /X} dy \ll \bigg(\frac{X}{j}\bigg)^\lambda.
\end{align}
Let $J$ be a parameter at our disposal to be chosen later. Then
\begin{align*}
\sum_{j=J+1}^\infty \frac{1}{j} \int_4^\infty j X^{-1}e^{-y j /X} \sum_{p_1 p_2 \le y} \e(jp_1 p_2 \alpha)dy \ll \sum_{j=J+1}^\infty \frac{1}{j} \frac{X}{j} \ll \frac{X}{J}.
\end{align*}
We can summarize this by saying that for any $J \ge 1$ we have
\begin{align} \label{Phionminor}
\Phi_{\mathbb{P}_2}(\rho \e (\alpha)) &= \sum_{j=1}^J \frac{1}{j} \sum_{p_1}\sum_{p_2} e^{-p_1 p_2 j /X} \e(j p_1 p_2 \alpha)dy + O \bigg(\frac{X}{J}\bigg) \nonumber \\
&= \sum_{j=1}^J \frac{1}{j} \int_4^\infty j X^{-1}e^{-y j /X }\sum_{p_1 p_2 \le y}  \e(j p_1 p_2 \alpha) dy + O \bigg(\frac{X}{J}\bigg).
\end{align}
Equipped with the fundamental estimate to bound the Weyl sum we may now prove the following result on the behavior of $\Phi(\rho \e (\alpha))$ in $\mathfrak{m}$.
\begin{lemma} \label{minorarclemmabound}
For $\alpha \in \mathfrak{m}$ one has that
\begin{align*}
\Phi_{\mathbb{P}_2}(\rho \e (\alpha)) \ll X (\log X)^{\frac{9}{4}-\frac{A}{8}},
\end{align*}
where $A>18$ is any real number. 
\end{lemma}
\begin{proof}
Fix $J$ be a parameter of our choice define the $y$-integral in \eqref{Phionminor} by
\begin{align} \label{expressionboundminor}
\mathfrak{I}(X,j):=\int_4^\infty j X^{-1}e^{-y j /X } S_2(j\alpha ,y) dy. %\sum_{p_1 p_2 \le x}  \e(j p_1 p_2 \alpha)
\end{align}
For each $j\leq J$, we employ Dirichlet's theorem (see\cite[Lemma 2.1]{vaughanbook}) to choose $a \in \Z$ and $q \in \N$ with $(a,q)=1$ such that 
\begin{align}
\bigg| j \alpha - \frac{a}{q} \bigg| \le q^{-1} X^{-1} (\log X)^A \quad \textnormal{and} \quad q < X(\log X)^{-A}.
\label{eq:lem_5_3_approx_dirichlet}
\end{align}
We now set $a_j:= a/(a,j)$ and $q_j:=jq/(a,j)$.
The definition of $\delta_q$ in \eqref{eq:def_delta_q} and \eqref{eq:lem_5_3_approx_dirichlet} imply that
%$\left|  \alpha - \frac{a_j}{q_j} \right| \le \delta_{q_j}$.
\begin{align}
\bigg|  \alpha - \frac{a_j}{q_j} \bigg| \le \delta_{q_j}.
\label{eq:lem_5_3_approx_dirichlet_2}
\end{align}
Since $\alpha \in \mathfrak{m}$, the definition of $\mathfrak{m}$ implies that $q_j > Q$, where $Q=(\log X)^A$. 
Moreover, Theorem~\ref{doublevinogradov} and \eqref{eq:lem_5_3_approx_dirichlet} imply that
\begin{align} \label{auxminorarcs01}
S_2(j\alpha, y) \ll\frac{y}{q^{\frac{1}{6}}}(\log y)^{\frac{7}{3}}+ y^{\frac{16}{17}} (\log y)^{\frac{39}{17}} + y^{\frac{7}{8}}q^{\frac{1}{8}}(\log y)^{\frac{9}{4}}.
\end{align}
%Recall that for any $\lambda >0$ \textcolor{red}{(Insert reference here)}
%\begin{align} \label{ibpminor}
%\int_{4}^\infty y^\lambda j X^{-1} e^{-y j /X} dy \ll \bigg(\frac{X}{j}\bigg)^\lambda.
%\end{align}
Integrating each of the three terms in \eqref{auxminorarcs01} and using \eqref {ibpminor} and yields
\begin{align} \label{threetemsminor}
E_1(X,j) = \int_4^\infty \frac{y}{q^{\frac{1}{6}}} (\log y)^{\frac{7}{3}} j X^{-1} e^{-yj/X} dy &\ll \frac{X}{{j{q^{1/6}}}}{\left( {\log \frac{X}{j}} \right)^{\frac{7}{3}}}, \nonumber \\
E_2(X,j) = \int_4^\infty y^{\frac{16}{17}} (\log y)^{\frac{39}{17}} j X^{-1} e^{-yj/X} dy &\ll \bigg(\frac{X}{j}\bigg)^{\frac{16}{17}} \bigg(\log \frac{X}{j} \bigg)^{\frac{39}{17}}, \nonumber \\
E_3(X,j) = \int_4^\infty y^{\frac{7}{8}}q^{\frac{1}{8}} (\log y)^{\frac{9}{4}} j X^{-1} e^{-yj/X} dy &\ll {\left( {\frac{X}{j}} \right)^{\frac{7}{8}}}{q^{1/8}}{\left( {\log \frac{X}{j}} \right)^{\frac{9}{4}}}.
\end{align}
If we now add all three terms in \eqref{threetemsminor} and use $jq\geq q_j > (\log X)^{A}$ on $E_1(X,j)$ and $q < X (\log X)^{-A}$ on $E_3(X,j)$, then we see that the integral in \eqref{expressionboundminor} is
\begin{align*} 
\mathfrak{I}(X,j)% &= E_1(X,j) + E_2(X,j) + E_3(X,j) \nonumber \\
&\ll X\frac{1}{{{j^{\frac{5}{6}}}}}{\left( {\log \frac{X}{j}} \right)^{\frac{7}{3}}}{(\log X)^{ - \frac{A}{6}}} + {\left( {\frac{X}{j}} \right)^{\frac{16}{17}}}{\left( {\log \frac{X}{j}} \right)^{\frac{39}{17}}} + \frac{X}{{{j^{\frac{7}{8}}}}}{(\log X)^{ - \frac{A}{8}}}\left( {\log \frac{X}{j}} \right)^{\frac{9}{4}}. \nonumber
\end{align*}
Next, summing over $j \in [1,J]$ we find that $\Phi_{\mathbb{P}_2}(\rho \e (\alpha))$ is
\begin{align}
&\ll
\sum\limits_{j = 1}^J \frac{1}{j}\bigg[ X\frac{1}{{{j^{\frac{5}{6}}}}}{{\left( {\log \frac{X}{j}} \right)}^{\frac{7}{3}}}{{(\log X)}^{ - \frac{A}{6}}} + {{\left( {\frac{X}{j}} \right)}^{\frac{16}{17}}}{{\left( {\log \frac{X}{j}} \right)}^{\frac{39}{17}}} + \frac{X}{{{j^{\frac{7}{8}}}}}{{(\log X)}^{ - \frac{A}{8}}}{\left( {\log \frac{X}{j}} \right)}^{\frac{9}{4}} \bigg] + O \bigg(\frac{X}{J}\bigg)\nonumber \\ 
&=
{\widetilde{S}_1}(J,X) + {\widetilde{S}_2}(J,X) + {\widetilde{S}_3}(J,X) + O \bigg(\frac{X}{J}\bigg).  \nonumber
\end{align}
Lastly,  we have for $j\leq J\leq \sqrt{X}$ that $\frac{1}{2}\log X \leq \log\frac{X}{j}\leq \log X$ and 
\begin{align*}
\widetilde{S}_1(J,X) 
= 
X{(\log X)^{ - \frac{A}{6}}}\sum\limits_{j = 1}^J {\frac{1}{{{j^{\frac{11}{6}}}}}{{\left( {\log \frac{X}{j}} \right)}^{\frac{7}{3}}}}  
%\ll
%X{(\log X)^{ - \frac{A}{6}}} {\left( {\log X} \right)}^{\frac{7}{3}}  \sum\limits_{j = 1}^J \frac{1}{{{j^{\frac{11}{6}}}}} \\
%&\ll
%X{(\log X)^{ - \frac{A}{6}}} {\left( {\log X} \right)}^{\frac{7}{3}}  \sum\limits_{j = 1}^\infty \frac{1}{{{j^{\frac{11}{6}}}}}
\ll X{(\log X)^{ - \frac{A}{6}}}(\log X)^{\frac{7}{3}}.
\end{align*}
Similarly, ${\widetilde{S}_2}(J,X) 
\ll {X^{\frac{16}{17}}}{(\log X)^{\frac{39}{17}}} $ and 
$
{\widetilde{S}_3}(J,X) \ll X{(\log X)^{ - \frac{A}{8}}}{(\log X)^{\frac{9}{4}}}
$.
Therefore
\begin{align}
\Phi_{\mathbb{P}_2}&(\rho \e (\alpha))  \ll X{(\log X)^{\frac{7}{3} - \frac{A}{6}}} + {X^{\frac{16}{17}}}{(\log X)^{\frac{39}{17}}} + X{(\log X)^{\frac{9}{4} - \frac{A}{8}}} + O \bigg(\frac{X}{J}\bigg) .
\end{align}
The result now follows by choosing $J=\sqrt{X}$ and taking $A>18$.
\end{proof}
%%%%%%%%%%%%%%%%%%%%%%%%%%%%%%%%%%%%%%%%%%%%%%%%%%%%%%%%%%%%%%%%%%%
\section{The non-principal major arcs} \label{sec:nonprincipalarcs}
%%%%%%%%%%%%%%%%%%%%%%%%%%%%%%%%%%%%%%%%%%%%%%%%%%%%%%%%%%%%%%%%%%%
Our first result is understanding the behavior of the following integral.
\begin{lemma} 
\label{lem:dirk'slemma}
Let $\gamma=\gamma_1+i\gamma_2\in\C$ with $\gamma_1>0$ with $\gamma_2 \ll \gamma_1 (\log(1/\gamma_1))^A$ for some $A>0$.
We then have as $\gamma_1\to 0$
\begin{align*}
\int_{2}^{\infty}  \frac{t\log\log t}{\log t} \exp(-\gamma t) dt  = \frac{\log\log(1/\gamma_1)}{\gamma^2\log(1/\gamma_1)}+O\left(\frac{\log\log(1/\gamma_1)}{\gamma_1^2\log^2(1/\gamma_1)}\right). 
\end{align*}
\end{lemma}
\begin{proof}
We split the interval $[2,\infty]$ into intervals $[2,d], [d,u]$ and $[u,\infty]$ with
\begin{align*}
d=\frac{1}{\gamma_1(\log(1/\gamma_1))^{A+1}} \quad \text{and} \quad u=\frac{(\log (1/\gamma_1))^{2}}{\gamma_1} .
\end{align*}
We now have that
\begin{align*}
\frac{\log\log t}{\log t}\geq 0 \text{ for } t\geq e \quad \text{and} \quad \lim_{t\to\infty} \frac{\log\log t}{\log t}=0.  
\end{align*}
%		Further $\frac{(\log\log t)}{\log t}$ is also bounded in the interval $[2,e]$. 
Thus there exist a constant $C>0$ such that $|\frac{\log\log t}{\log t}|\leq C$  for all $t\geq 2$. Since $\gamma_1>0$ and $|\exp(z)| =\exp(\real(z))$, the integral over $[2,d]$ can be estimated as
\begin{align*}
\left|\int_{2}^{d}  \frac{t\log\log t}{\log t} \exp(-\gamma t)  dt\right|
&\leq 
C\int_{2}^{d} t\exp(-\gamma_1 t) dt 
\leq 
C\int_{2}^{d} t\, dt \ll \frac{1}{\gamma_1^{2}(\log(1/\gamma_1))^{2A+2}}.
\end{align*}
For the integral over $[u,\infty]$, we use that $\frac{\gamma_1}{2}t\leq \exp(\frac{\gamma_1}{2}t)$ for $t\geq 0$.		
Furthermore $\frac{\log\log t}{\log t}$ is monotonically decaying for $t$ large enough. 
This gives
\begin{align*}
\left|\int_{u}^{\infty} \frac{t\log\log t}{\log t}  \exp(-\gamma t) dt\right|
&
\leq 
C\frac{\log\log u}{\log u}\frac{2}{\gamma_1}\int_{u}^{\infty} \left(\frac{\gamma_1}{2}t\right) \exp(-\gamma_1 t)dt \nonumber \\
&=
\frac{C\log\log u}{\log u}\frac{2}{\gamma_1}\int_{u}^{\infty} \exp\left(-\frac{\gamma_1}{2} t\right)dt\\
%&=
%\frac{C\log\log u}{\log u}\frac{4}{\gamma^2_1} \exp\left(-\frac{\gamma_1}{2} u\right) 
&\ll \frac{\log\log(1/\gamma_1)}{\gamma_1^2(\log(1/\gamma_1))^3}.
\end{align*}
We have used on the last line that $\gamma_1 u = \log^2(1/\gamma_1)$.
For the computation of the remaining integral, we use the observation that for $t\in[d,u]$ one has that
\begin{align*}
\frac{\log\log t}{\log t} = \frac{\log\left(\log(1/\gamma_1)+\log(\gamma_1 t)\right)}{\log(1/\gamma_1)+\log(\gamma_1 t)} 
&=
\frac{\log\log(1/\gamma_1)+\log(1+\frac{\log (\gamma_1  t)}{\log(1/\gamma_1)})}{\log(1/\gamma_1)(1-\frac{\log \gamma_1  t}{\log(1/\gamma_1)})}\\
&=
\frac{\log\log(1/\gamma_1)}{\log(1/\gamma_1)}+O\left(\frac{\log\log(1/\gamma_1)^{2}}{\log(1/\gamma_1)^{2}}\right).  
\end{align*}
This implies together with a similar bound as above that 
\begin{align*}
\int_{d}^{u} \frac{t\log\log t}{\log t} \exp(-\gamma t) dt &=
\frac{\log\log(1/\gamma_1)}{\log(1/\gamma_1)}\int_{d}^{u} t\exp(-\gamma t) dt  + O\left(\frac{\log\log(1/\gamma_1)^{2}}{\log(1/\gamma_1)^{2}}\int_{d}^{u} t\exp(-\gamma_1 t) dt\right)\\
%			=\,&
%			\frac{\log\log(1/\gamma_1)}{\log(1/\gamma_1)}\int_{d}^{u} t\exp(\gamma t) dt+ O\left(\frac{\log\log(1/\gamma_1)^{2}}{\log(1/\gamma_1)^{2}}\int_{d}^{u} t\exp(-\gamma_1 t) dt\right)\\
&=
\frac{\log\log(1/\gamma_1)}{\log(1/\gamma_1)}\int_{d}^{u} t\exp(-\gamma t) dt+ O\left(\frac{\log\log(1/\gamma_1)^{2}}{\gamma_1^2\log(1/\gamma_1)^{2}}\right).         
\end{align*}
Since $\gamma d = \log^{-1}(1/\gamma_1)$ by construction, we get
\begin{align}
\int_{d}^{u} t\exp(-\gamma t) dt
&=
\frac{\left(\gamma  d +1\right) e^{-\gamma d}}{\gamma^{2}}- \frac{\left(\gamma  u +1\right) e^{-\gamma u}}{\gamma^{2}}\nonumber\\
&=
\frac{1}{\gamma^{2}} + O\left(\frac{1}{\gamma^{2}\log(1/\gamma_1)}\right)
+O\left(\frac{\left(\gamma  u +1\right) e^{-\gamma_1 u}}{\gamma^{2}}\right).
\end{align}
This completes the proof of the lemma.
\end{proof}
We now study a finite sum over products of primes satisfying a congruence and write this sum as an integral involving the logarithmic integral $\operatorname{Li}$.
\begin{theorem}
\label{lem:Ugly_A_mod_q}
Let $q,\ell\in\N$ be given  with $(\ell,q)=1$.
We then have as $t\to\infty$
\begin{align}
A(t)
:=
\sum_{\substack{p_1 p_2\leq t\\p_1 p_2\equiv\ell\modu q}} 1
=\, &
\frac{2}{\varphi(q)}
\int_2^{\sqrt t }
(\log \log u + M +O((\log u)^{-C})\left(\operatorname{Li}(t/u) -\frac{t}{u \log (t/u) }\right) du \nonumber\\
&-\frac{\operatorname{Li}^2(\sqrt{t})}{\varphi(q)}+\frac{2\sqrt{t}}{\varphi(q)}\operatorname{Li}(\sqrt{t})  \left( {\log \log t - \log 2 + M } \right)
+ O\left(t (\log t)^{-C} \right),
\end{align}
where $C>0$ is arbitrarily large and where $\varphi$ is the Euler totient function.
\end{theorem}
\begin{proof}
Observe that we can write $A(t) = \card \mathcal{A}$ with
\begin{align*}
\mathcal{A}
=
\{ (p_1, p_2): p_1,p_2 \textnormal{ prime}, \  p_1 p_2 \le t,\textnormal{ and } p_1 p_2 \equiv \ell \modu q\}.
\end{align*}	
We now write $A(t) = A_1(t) + A_2(t) - A_3(t)$ with $A_i(t) =\card \mathcal{A}_i$ for $i=1,2,3$, where
\begin{align*}
\mathcal{A}_1 =\{(p_1, p_2)\in\mathcal{A}: p_1\leq \sqrt{t}\}, \
\mathcal{A}_2 =\{(p_1, p_2)\in\mathcal{A}: p_2\leq \sqrt{t}\} \ \text{ and } \
\mathcal{A}_3 =\{(p_1, p_2)\in\mathcal{A}: p_1,p_2\leq \sqrt{t}\}.
\end{align*}
An illustration of $A_1(t)$, $A_2(t)$ and $A_3(t)$ can be found in Figure~\ref{fig1}.
\begin{figure}[h]
\centering
\begin{tikzpicture}
\draw[thick,->] (-1,0) -- (6,0) node[anchor=north west] {$p_1$};
\draw[thick,->] (0,-1) -- (0,6) node[anchor=south east] {$p_2$};
\draw (.5,5.5) .. controls (1,1.3) and (1.3,1) .. (5.5,.5) ;
\draw[thick] (0,1.5) -- (1.7,1.5);
\draw[thick] (1.7,0) -- (1.7,1.5);
\draw (1,0) node[anchor=north west] {$\sqrt{t}$};
\draw (0,1) node[anchor=south east] {$\sqrt{t}$};

\draw[blue] (1.5,0) -- (1.5,1.72);
\draw[blue] (1.3,0) -- (1.3,1.98);
\draw[blue] (1.1,0) -- (1.1,2.37);
\draw[blue] (.9,0) -- (.9,3);
\draw[blue] (.7,0) -- (.7,3.8);
\draw[blue] (.5,0) -- (.5,5.2);
\draw[blue] (.3,0) -- (.3,5.5);
\draw[blue] (.1,0) -- (.1,5.5);

\draw[red] (0,.1) -- (5.5,.1);
\draw[red] (0,.3) -- (5.5,.3);
\draw[red] (0,.5) -- (5.2,.5);
\draw[red] (0,.7) -- (3.8,.7);
\draw[red] (0,.9) -- (3,.9);
\draw[red] (0,1.1) -- (2.37,1.1);
\draw[red] (0,1.3) -- (1.98,1.3);

\draw (5.5,1) node[anchor=north west] {$A_2$};
\draw (1.2,5.6) node[anchor=south east] {$A_1$};
\draw (2.2,2) node[anchor=east] {$A_3$};
%\draw (-2.3,1) node[anchor=north west] {$S(\gamma,M,N)$};
\draw [thin, ->] (2.3,1.7) -- (1.5,1);
\end{tikzpicture}
\caption{Hyperbola representation for the regions $A_i$ with $i=1,2,3$.} 
\label{fig1}
\end{figure}
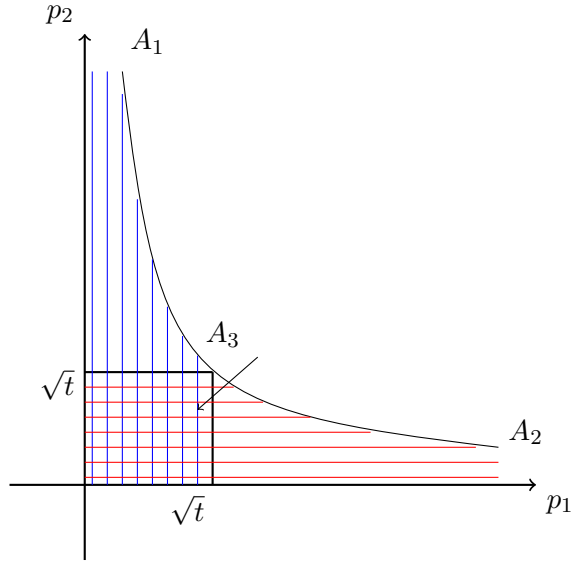

Let us start with the square $A_3(t)$. Using the definition of $\mathcal{A}_3$, we can write
\begin{align}
A_3(t) 
&= 
\sum_{\substack{p_1 \le \sqrt{t} \\ q \not\equiv 0 \modu p_1}} \card\{p_2 \le \sqrt{t}, p_2 \equiv \ell \bar p_1 \modu q \},
\label{eq:A_3_fubini}		
\end{align}
where $\bar p_1$ is the inverse of $p_1$ modulo $q$.
This inverse exists since $(p_1,q)=1$, which follows immediately from the equality $\ell = p_1p_2 +kq$ with $k\in\Z$ and the assumption $(\ell,q)=1$.	
For the summand in $A_3$ we note that $(\ell \bar p_1 ,q)=1$ and we apply Siegel-Walfisz's theorem so that
\begin{align}
\card\{p_2 \le \sqrt{t}, p_2 \equiv \ell \bar p_1 \modu q \} = \pi(\sqrt{t}; q, \ell \bar p_1) = \frac{\operatorname{Li}(\sqrt{t})}{\varphi(q)} + O\bigg(\sqrt{t} \exp \bigg(-\frac{C_N}{2} (\log \sqrt{t})^{\frac{1}{2}} \bigg)\bigg).
\label{eq:A_3_siegel_Wal}
\end{align}
Since the expression on the right-hand side of \eqref{eq:A_3_siegel_Wal} is independent of $p_1$, the prime number theorem with the standard zero-free region
and that $\operatorname{Li}(t)\sim t/\log t$ imply
\begin{align*}
A_3(t)
&= 
\left(\frac{\operatorname{Li}(\sqrt{t})}{\varphi(q)} +  O\bigg(\sqrt{t} \exp \left(-\frac{C_N}{2}(\log \sqrt{t})^{1/2} \right)\bigg)\right)  \bigg(\sum_{p_1 \le \sqrt{t}} 1\bigg) \\
&=
\left(\frac{\operatorname{Li}(\sqrt{t})}{\varphi(q)} +  O(\sqrt{t} (\log t)^{-C}) \right)(\operatorname{Li}(\sqrt{t}) +  O(\sqrt{t} \exp (-C(\log \sqrt{t})^{1/2} )))\\
&= 
\frac{\operatorname{Li}^2(\sqrt{t})}{\varphi(q)} + O(t (\log t)^{-C} ).
\end{align*}
Now we move to $A_1$. Observe that by symmetry $A_1=A_2$ and thus it is sufficient to consider $A_1$. 
We have
\begin{align*}
A_1(t) = \sum_{p_1 \le \sqrt{t}} \card \{ p_2 \le t/p_1 : p_2 \equiv \ell \bar p_1 \modu q\}.
%=
%\sum_{\substack{p_2\leq t/p_1 \\ {p_2\equiv \ell\bar{p}_1}(\bmod{q})}} \sqrt{t} \exp (-\frac{C_N}{2} (\log t)^{1/2})    
\end{align*}
The summand can be found again by Siegel-Walfisz's theorem 
\begin{align*}
\card \{ p_2 \le t/p_1 : p_2 \equiv \ell \bar p_1 \modu q\} 
= 
\pi(t/p_1; q, \ell \bar p_1) 
&= 
\frac{\operatorname{Li}(t/p_1)}{\varphi(q)} + O\bigg(\frac{t}{p_1} \exp \bigg(-\frac{C_N}{2} \bigg(\log \frac{t}{p_1} \bigg)^{\frac{1}{2}} \bigg)\bigg) \nonumber \\
&= 
\frac{\operatorname{Li}(t/p_1)}{\varphi(q)}+  O\left( \frac{t/p_1}{(\log(t/p_1))^{C+1}}\right).
\end{align*}
Using that $\log \sqrt{t} \leq \log(t/p_1) \leq \log t$, we can estimate the sum over the error term as
\begin{align*}
\sum_{p_1 \le \sqrt{t}} \frac{t/p_1}{(\log(t/p_1))^{C+1}}
\ll 
\frac{t}{(\log t)^{C+1}}\sum_{p_1 \le \sqrt{t}} \frac{1}{p_1}
\ll
\frac{t\log\log t}{(\log t)^{C+1}}
\ll
\frac{t}{(\log t)^{C}}.
\end{align*}
This implies that
\begin{align*}
A_1(t) = \frac{1}{\varphi(q)}\sum_{p_1 \le \sqrt{t}} \operatorname{Li}(t/p_1) +	O\left(\frac{t}{(\log t)^{C}}\right).
\end{align*}
We now apply Abel's summation formula to the sum in $A_1(t)$.  
We set $B(t):= \sum_{n\leq t} b_n$ with 
\begin{align*}
b(n) = \frac{{\mathbf{1}(n \in \mathbb{P})}}{n}, \text{ and} \quad f(n) = n\operatorname{Li}(t/n), \quad f'(n) = \operatorname{Li}(t/n) -\frac{t}{n \log (t/n) },
\end{align*}
where $\mathbf{1}(n \in \mathbb{P})$ is 1 when $n$ is a prime and 0 otherwise. 
Then we see that
\begin{align*}
S(t) 
&:= 
\sum_{p \le \sqrt{t}} 	\operatorname{Li}(t/p) 
= 
\sum_{n \le \sqrt{t}} b_n n\operatorname{Li}(t/n) = B(\sqrt t )f(\sqrt t ) - B(1)f(1) - \int_2^{\sqrt t } {B(u)f'(u)du}.	
\end{align*}
We now look at the three summands in this equation. 
First notice that $B(x)= 0$ for $t<2$ and thus $B(1)f(1)=0$. 
For the other two summands, we use Landau's zero-free region form of Mertens' estimate (\cite{landauhandbook})
\begin{align*}
B(x) = \sum_{n \leqslant x} b(n) = \sum_{p \leqslant x} \frac{1}{p} = \log \log x + M + O(\exp(-(\log x)^{\frac{1}{14}})).
\end{align*}
This then implies that
\begin{align*}
B(\sqrt t )f(\sqrt t ) 
&= 
( {\log \log \sqrt t  + M + O((\log t)^{-C})}) \sqrt{t}\operatorname{Li}(\sqrt{t}) \nonumber \\
&= 
\sqrt{t}\operatorname{Li}(\sqrt{t})  \left( {\log \log t - \log 2 + M } \right) + O(t(\log t)^{-C}).
\end{align*}
Finally
\begin{align*}
\int_2^{\sqrt t } {B(u)f'(u)du}
&=
\int_2^{\sqrt t }
(\log \log u + M +O(\left(\log u\right)^{-C}))\left(\operatorname{Li}(t/u) -\frac{t}{u \log (t/u) }\right) du.
%			&=
%			\int_2^{\sqrt t }
%			\left(\log \log u + M \right)\left(\operatorname{Li}(t/u) -\frac{t}{u(\log (t/u))}\right)\,du
%			+
%			O\left(\frac{t}{\log t}\int_2^{\sqrt t } \frac{t}{u(\log u)^{C+1}} \right)\\
%&=
%\int_2^{\sqrt t }
%\left(\log \log u + M \right)\left(\operatorname{Li}(t/u) -\frac{t}{u \log (t/u) }\right) du
%+ O(t(\log t)^{-C}).
\end{align*}
Putting these pieces back together ends the proof.
\end{proof}
We can apply our two previous lemmas to find the asymptotic of an infinite exponential sum over a product of primes that satisfies the congruence of our previous lemma.
\begin{lemma} \label{lem:main_sum_exp_gamma}
Let $\gamma = \gamma_1+i\gamma_2$ with $\gamma_1 >0$ and $\gamma_2 \ll \gamma_1 (\log(1/\gamma_1))^A$ for some $A>0$. 
Moreover let $q\in\N$ with $q \ll (\log(1/\gamma))^A$ and define
\begin{align*}
U(\gamma, \ell, q) := \sum_{p_1 p_2 \equiv \ell \bmod q} \exp(-p_1 p_2 \gamma), 
\label{eq:def_U_gamma_ell_q}   
\end{align*}
with $(\ell,q)=1$. Then one has that
\begin{align*}
U(\gamma, \ell, q)
=
\frac{2}{\varphi(q)}\frac{\log\log(1/\gamma)+M}{\gamma\log(1/\gamma)}
+O\bigg(\frac{1}{\varphi(q)\gamma_1}\bigg(\frac{\log\log(1/\gamma_1)}{\log(1/\gamma_1)}\bigg)^2\bigg),   
\end{align*}
where $\varphi(q)$ is the Euler totient function.
\end{lemma}
\begin{proof}
First, we apply Abel's summation to $U(\gamma, \ell, q)$.
For this, we define $A(t)= \sum_{n\leq t} a_n$ with 
\begin{align}
a_n 
= 
\begin{cases}
2, \quad &\mbox{if $n=p_1p_2$ with $p_1\neq p_2$ and $n \equiv \ell \modu q$,} \\
1, \quad &\mbox{if $n=p^2$ and $n \equiv \ell \modu q$,} \\
0, \quad &\mbox{otherwise.} 
\end{cases}
\end{align}
Note that this $A(t)$ agrees with $A(t)$ in Theorem~\ref{lem:Ugly_A_mod_q}.
Moreover, set $f(n) = \exp(-\gamma n)$. We get
\begin{align}
U(\gamma, \ell, q) 
=
\sum_{n\leq t} a_n f(n)
= 
A(x)f(x)\big|_{x=2}^{\infty} - \int_2^\infty A(t) f'(t)dt =  \gamma \int_2^\infty A(t) \exp(-\gamma t)dt.
\label{eq:U_with_A}
\end{align}
We used in the last equality that $A(t) = 0$ for $t\leq 2$ and $A(t)\leq 2t$ otherwise.
Inserting Theorem~\ref{lem:Ugly_A_mod_q}, we see that
\begin{align*}
U(\gamma, \ell, q)
=
\gamma \int_2^\infty D(t) \exp(-\gamma t)dt
+ O\left(\gamma \int_2^\infty t (\log t)^{-C}\exp(-\gamma t)dt \right)
\end{align*}
where $D$ is given by
\begin{align*}
D(t)
=\, &
\frac{2}{\varphi(q)}
\int_2^{\sqrt t }
\big(\log \log u + M +O((\log u)^{-C})\big)\left(\frac{t}{u \log (t/u) }-\operatorname{Li}(t/u)\right) du \\
&-\frac{\operatorname{Li}^2(\sqrt{t})}{\varphi(q)}+\frac{2\sqrt{t}}{\varphi(q)}\operatorname{Li}(\sqrt{t})  \left( {\log \log t - \log 2 + M } \right).
\end{align*}
Using Lemma~\ref{lem:eq:integral_with_gamma3} with $C> 2A+2$, the assumption on $\gamma_2$ and that $\varphi(q)\leq q$, we get 
\begin{align*}
\left|\gamma \int_2^\infty t (\log t)^{-C}\exp(-\gamma t)\,dt\right|
&\leq
|\gamma| \int_2^\infty t (\log t)^{-C}\exp(-\gamma_1 t)\,dt
\ll
|\gamma| \frac{1}{\gamma_1^2 (\log(1/\gamma_1))^C}\\
&\ll
\frac{2\gamma_1 (\log(1/\gamma_1))^A}{\gamma_1^2 (\log(1/\gamma_1))^{2A+2}}
\ll
\frac{1}{\varphi(q) \gamma_1 (\log(1/\gamma_1))^2}.
\end{align*}
This implies that
\begin{align*}
U(\gamma, \ell, q)
=
\gamma \int_2^\infty D(t) \exp(-\gamma t)dt
+O\left(\frac{1}{\varphi(q)\gamma_1 (\log(1/\gamma_1))^2}\right).
\end{align*}
Since $\operatorname{Li}(t) = t/\log t + O(t/(\log t)^2)$, we have 
\begin{align*}
D(t)
\ll 
\int_2^{\sqrt t }
\left(\log \log t + M + O((\log u)^{-C})\right)\frac{tdu}{u(\log (t/u))^2} +\frac{t \log\log t}{\log t}
\ll
\frac{t\log\log t}{\log t}
\end{align*}
Thus, we get with partial integration that
\begin{align*}
\gamma \int_2^\infty D(t) \exp(-\gamma t)dt
=
\int_2^\infty D'(t) \exp(-\gamma t)dt +O(1).
\end{align*}		
For the derivative of $D$ we have
\begin{align*}
D'(t)
&=
\frac{2}{\varphi(q)}\int_2^{\sqrt t }
\left(\log \log u + M + O((\log u)^{-C})\right)\frac{\partial}{\partial t}\left(\frac{t}{u \log (t/u) } -\operatorname{Li}(t/u)\right) du\\
&\quad +
\frac{2}{\varphi(q)2\sqrt{t}}\left(\log \log \sqrt{t} + M+O((\log \sqrt{t})^{-C}))\right)\left( \frac{\sqrt{t}}{\log\sqrt{t}}-\operatorname{Li}(\sqrt{t})\right)
-
\frac{2 \operatorname{Li}(\sqrt{t})}{\varphi(q)\sqrt{t}\log t} \\
&\quad +
\frac{1}{\varphi(q)}\frac{\operatorname{Li}(\sqrt{t}) (\log\log t-\log 2 + M)}{\sqrt{t}} 
+	\frac{2}{\varphi(q)}\frac{\operatorname{Li}(\sqrt{t})}{\sqrt{t} \log t} +\frac{2}{\varphi(q)}\frac{\log\log t-\log 2+ M}{\log t}\\
&=
-\frac{2}{\varphi(q)}\int_2^{\sqrt t }
\frac{\log \log u + M + O((\log u)^{-C})}{u\log^2(t/u)} du \nonumber \\
&\quad +\frac{4}{\varphi(q)}\frac{\log\log t-\log 2+ M}{\log t}+ O\left(\frac{1}{\varphi(q)(\log t)^{C+2}}\right)
\\
&=
\frac{2}{\varphi(q)}\left( \frac{\log(t/2)\log\log(t/2) + \log 2 (\log\log 2)}{\log(t/2) \log t}  - \frac{\log\log \sqrt{t} }{\log\sqrt{t}}+ 	\frac{M}{\log(t/2)} -\frac{M}{\log \sqrt{t}}\right) \\
&\quad +\frac{4}{\varphi(q)}\frac{\log\log t-\log 2+ M}{\log t}+ O\left(\frac{\log\log t}{\varphi(q)(\log t)^2}\right)\\
&=
\frac{2}{\varphi(q)}\frac{\log\log t+ M}{\log t}+ O\left(\frac{\log\log t}{\varphi(q)(\log t)^2}\right).
\end{align*}	
We thus arrive at
\begin{align*}
U(\gamma, \ell, q)	
=
\int_2^\infty \left(	\frac{2}{\varphi(q)}\frac{\log\log t+ M}{\log t}+ O\left(\frac{\log\log t}{\varphi(q)(\log t)^2}\right)\right)e^{-\gamma t}dt
+O\left(\frac{1}{\varphi(q)\gamma_1 (\log(1/\gamma_1))^2}\right).
\end{align*}
Applying a small variation of Lemma~\ref{lem:dirk'slemma} finally completes the proof.	
\end{proof}
Now we can now employ Lemma~\ref{lem:main_sum_exp_gamma} to write an asymptotic expression for $\Phi_{\mathbb{P}_2}(\rho \e(\alpha))$ involving the Ramanujan sum $\sum_{\ell=1, (\ell,q)=1}^{q} \e(\frac{a\ell}{q})$, which is a key step in the bound we are looking for for the non-principal major arcs.
\begin{lemma}\label{lemma5.1}
Let $\alpha\in\R$ and $A>0$ be given.
Further, let $a\in\Z$, $q\in\N$ with 
\begin{align*}
(a,q)=1, \quad q\leq (\log X)^A \quad \text{ and } \quad \left|\alpha - \frac{a}{q}\right|\leq q^{-1}X^{-1}(\log X)^A.
\end{align*}
Then there exists $X_0(A)$ such that we have for all $X>X_0(A)$ 
\begin{align}
\Phi_{\mathbb{P}_2}(\rho \e(\alpha))
=
2X\frac{\log\log X}{(1-2\pi i (\alpha-\frac{a}{q}) X)\log X}\sum_{j\leq\sqrt{X}}\frac{S^{*}(q_j,a_j)}{j^2\varphi(q_j)}+O\bigg(\frac{X \log \log \log\log X}{\log X}\bigg),
\label{eq:Phi_on_major_pain}
\end{align}
where $S^*$ is defined by
\begin{align*}
S^{*}(q,a):=\sum_{\substack{1 \le \ell \le q\\ (\ell,q)=1}} \e\left(\frac{a\ell}{q}\right)% \quad \textnormal{and} \quad q_j=q/(q,j)
\end{align*}
and $q_j=q/(q,j)$ as well as $a_j=aj/(q,j)$.
\end{lemma}

\begin{proof}
We define $\gamma = \frac{1}{X}- 2\pi i \beta$ with $\beta=\alpha-\frac{a}{q}$.
Then \eqref{Phionminor} with $J=\sqrt{X}$ implies that 	
\begin{align}\label{firstexpression}
\Phi_{\mathbb{P}_2}\left(\rho\e\left(\alpha\right)\right)
&=
\Phi_{\mathbb{P}_2}\left(\rho\e\left(\frac{a}{q}+\beta\right)\right)
=
\sum_{j=1}^{\sqrt{X}}\frac{1}{j} \sum_{p_1} \sum_{p_2} \e\left(\frac{ajp_1p_2}{q}\right)\exp(-jp_1p_2\gamma)+O(X^{1/2}).  %\nonumber\\
%&=
%\sum_{j=1}^{\sqrt{X}}\frac{1}{j} \sum_{p_1} \sum_{p_2} \e\left(\frac{a_jp_1p_2}{q_j}\right)\exp(-jp_1p_2\gamma)+O(X^{1/2}).%\nonumber\\
\end{align}
We now split the sum over $p_1$ and $p_2$ into the following four cases
\begin{align*}
\mathcal{P}_1 &=\{(p_1,p_2);\, \ q_j \not \equiv 0 \modu p_1 \ \text{ and } q_j \not \equiv 0 \modu p_2 \}, \nonumber \\
\mathcal{P}_2 &=\{(p_1,p_2);\, \ q_j \equiv 0 \modu p_1 \ \text{ and } q_j \equiv 0 \modu p_2 \}, \nonumber \\
\mathcal{P}_3
&=\{(p_1,p_2);\, \ q_j \equiv 0 \modu p_1 \ \text{ and } q_j \not \equiv 0 \modu p_2 \}, \nonumber \\
\mathcal{P}_4 &=\{(p_1,p_2);\, \ q_j \not \equiv 0 \modu p_1 \ \text{ and } q_j \equiv 0 \modu p_2 \},
\end{align*}
and for $1  \leq k\leq 4$ we define the following sum
\begin{align*}
S_k:=\sum_{\mathcal{P}_k}\e\left(\frac{a_jp_1p_2}{q_j}\right)\exp(-jp_1p_2\gamma).
\end{align*}
We will see that the main contribution comes form $S_1$, but first, we look at the other sums.
Since the number of prime factors of $q_j$ is $O(\log(q_j))$, we get that $S_2 \ll (\log q_j)^2$.
Further, by symmetry $S_3=S_4$. 
Bounding each summand trivially and using again that the number of prime factors of $q_j$ is $O(\log(q_j))$,
we deduce
\begin{align}
S_3
&\ll
\sum_{q_j \equiv 0 \modu p_1}\sum_{q_j \not \equiv 0 \modu p_2}\exp(-jp_1p_2\gamma_1)
=
\sum_{q_j \equiv 0 \modu p_1}\sum_{p_2}\exp(-jp_1p_2\gamma_1) +O(\log^2(q_j)).
\label{eq:S3_ugly_sum}
%	&=
%	\sum_{p_1|q_j} \frac{1}{j\gamma_1p_1 \log(1/(j\gamma_1p_1))} +O(\log^2(q_j))
%	\ll \frac{X\log(q_j)}{j \log X} +O(\log^2(q_j)).
\end{align}
We now have 
\begin{align}
\sum_{p_2}\exp(-jp_1p_2\gamma_1)
=
\frac{1	}{jp_1\gamma_1\log(jp_1\gamma_1)} + 
O\left(\frac{1}{jp_1\gamma_1}\frac{\log\log(1/jp_1\gamma_1)}{\left(\log(1/jp_1\gamma_1)\right)^2}\right)
\ll
\frac{X}{jp_1\log X}.   
\label{eq:ugly_sum_in_one_prime_case}  
\end{align}
The proof of \eqref{eq:ugly_sum_in_one_prime_case} is similar to the proof of Lemma~\ref{lem:main_sum_exp_gamma}, but much less involved and can also be found in \cite[p. 11-12]{gafniprimepowers}.
We thus do not state the proof of \eqref{eq:ugly_sum_in_one_prime_case}.  
Inserting \eqref{eq:ugly_sum_in_one_prime_case} into \eqref{eq:S3_ugly_sum} gives
\begin{align*}
S_3
&\ll
\sum_{q_j \equiv 0 \modu p_1}\sum_{q_j \not\equiv 0 \modu p_2}\exp(-jp_1p_2\gamma_1)
=
\frac{X}{j\log X}\sum_{q_j \equiv 0 \modu p_1}\frac{1}{p_1}   +O((\log q_j)^2).
\end{align*}
Denote by $m$ the number of distinct prime factors of $q_j$. Then Mertens' theorem implies
\begin{align*}
\sum_{\substack{p_1 \\ q_j \equiv 0 \modu p_1}}\frac{1}{p_1}
\leq 
\sum_{p\leq m} \frac{1}{p}
= 
\log\log m + M + o(1)
=
O(\log\log m).
\end{align*}
Since $q_j$ has at most $O(\log q_j)$ primes factors and $q_j \ll (\log X)^A$, we deduce that
\begin{align*}
S_3
&\ll
\frac{X\log\log\log\log X }{j\log X} + O((\log q_j )^2).
\end{align*}
Inserting this as well as the bound for $S_2$ into \eqref{firstexpression} we get 
\begin{align*}
\Phi_{\mathbb{P}_2}\left(\rho\e\left(\alpha\right)\right)
&=
\sum_{j=1}^{\sqrt{X}}\frac{1}{j} \sum_{q_j \not \equiv 0 \modu p_1} \sum_{q_j \not \equiv 0 \modu p_2} 	 \e\left(\frac{a_jp_1p_2}{q_j}\right)\exp(-jp_1p_2\gamma)+O\bigg(\frac{X \log \log \log\log X}{\log X}\bigg)\nonumber\\
&=
\sum_{j=1}^{\sqrt{X}}\frac{1}{j}\bigg\{\sum_{\substack{\ell = 1\\ (\ell,q_j)=1}}^{q_j} \e\left(\frac{a_j\ell}{q_j}\right)\sum_{p_1 p_2\equiv\ell\bmod q_j} e^{-j p_1 p_2\gamma} \bigg\} 
+O\bigg(\frac{X \log \log \log\log X}{\log X}\bigg)\nonumber\\
&=
\sum_{j=1}^{\sqrt{X}}\frac{1}{j}\bigg\{\sum_{\substack{\ell = 1\\ (\ell,q_j)=1}}^{q_j} \e\left(\frac{a_j\ell}{q_j}\right)U(\gamma, \ell, q) \bigg\}  
+O\bigg(\frac{X \log \log \log\log X}{\log X}\bigg), 
\end{align*}
with $U(\gamma, \ell, q)$ as in \eqref{eq:def_U_gamma_ell_q}. 
Since the leading term of $U(\gamma, \ell, q)$ in Lemma~\ref{lem:main_sum_exp_gamma} does not dependent on $\ell$, we get
\begin{align*}
\sum_{j=1}^{\sqrt{X}}\frac{1}{j}\bigg\{\sum_{\substack{\ell = 1\\ (\ell,q_j)=1}}^{q_j} \e\left(\frac{a_j\ell}{q_j}\right)U(\gamma, \ell, q) \bigg\}  
&=
\sum_{j=1}^{\sqrt{X}}\frac{1}{j^2} 	S^{*}(q_j,a_j) 	\frac{2}{\varphi(q_j)}\frac{\log\log(1/(j\gamma))+M}{\gamma\log(1/(j\gamma))}\\
&\quad +
\sum_{j=1}^{\sqrt{X}}\frac{1}{j^2} S^{*}(q_j,a_j)  O\bigg(\frac{1}{\varphi(q_j)\gamma_1}\bigg(\frac{\log\log(1/(j\gamma_1))}{\log(1/(j\gamma_1))}\bigg)^2\bigg).  
\end{align*}
Since $| S^{*}(q_j,a_j) |\leq \varphi(q_j)$ and $1/\gamma_1 =X$ and $\sqrt{X}\leq X/j\leq X$, we get
\begin{align}\label{eq:needs_a_good_name}
\Phi_{\mathbb{P}_2}\left(\rho\e\left(\alpha\right)\right)
&=
\sum_{j=1}^{\sqrt{X}}\frac{1}{j^2} 	S^{*}(q_j,a_j) 	\frac{2}{\varphi(q_j)}\frac{\log\log(1/(j\gamma))+M}{\gamma\log(1/(j\gamma))} \\ 
&\quad +
O\bigg(X\bigg(\frac{\log\log X}{\log X}\bigg)^2\sum_{j=1}^{\sqrt{X}}\frac{1}{j^2}\bigg) + O\bigg(\frac{X \log \log \log\log X}{\log X}\bigg). \nonumber
\end{align}
Furthermore, we have
\begin{align*}
\frac{1}{\log(1/(j\gamma))}
&=
\frac{1}{\log(\frac{1}{j(1/X-2\pi i \beta)})}
=
\frac{1}{\log X -\log j -\log(1-2\pi i X\beta)}\\
&=
\frac{1}{\log X} \frac{1}{1 -\frac{\log j +\log(1-2\pi i X\beta)}{\log X}}
=
\frac{1}{\log X} +O\left( \frac{\log j +\log|1-2\pi i X\beta|}{(\log X^2)}  \right)\\
&=
\frac{1}{\log X} +O\left( \frac{\log j +\log\log X}{(\log X)^2}  \right)
\end{align*}
as well as
\begin{align*}
\log\log(1/(j\gamma))
=
\log\log X +O\left(  \frac{\log j +\log\log X}{\log X}  \right).
\end{align*}
Inserting these two identities into \eqref{eq:needs_a_good_name} completes the proof.
\end{proof}
\begin{lemma}
\label{lem:Vaugah_constanot_major}
One has that
\begin{align}
\sum_{j\leq\sqrt{X}}\frac{S^{*}(q_j,a_j)}{j^2\varphi(q_j)}
=
\zeta(2)\frac{\prod_{p|q}(-p)}{q^2}
+O(X^{-1/2})
\label{eq:Phi_on_major_pain_less}
\end{align}
with $q_j$ and $S^{*}(q_j,a_j)$ as in Lemma \textnormal{\ref{lemma5.1}}.
\end{lemma}
\begin{proof}
It is well known that $S^{*}(q_j,a_j)= (-1)^j\mu(q_j)$.
The remaining steps are straight forward and we thus omit them.
\end{proof}
Lemma~\ref{lemma5.1} immediately implies 
\begin{corollary}
\label{cor:lemma5.1}
We have for $\alpha\in\mathfrak{M}(1,0)$ with $\mathfrak{M}(a,q)$ as in \eqref{eq:def_M(q,a)} that 
\begin{align*}
\Phi_{\mathbb{P}_2}(\rho \e(\alpha))
=
2\zeta(2)\frac{X\log\log X}{(1-2\pi i \alpha X)\log X}+O\bigg(\frac{X \log \log \log\log X}{\log X}\bigg).
\end{align*}
\end{corollary}
\begin{proof}
We have $\alpha\in\mathfrak{M}(1,0)$ is equivalent to $|\alpha|\leq X^{-1}(\log X)^A$.
Thus can use Lemma~\ref{lemma5.1} with $q=1$, $a=0$. 
We have in this case $q_j =1$ for all $j$ and $S^{*}(1,0) =1$. 
Thus we obtain
\begin{align*}
\Phi_{\mathbb{P}_2}(\rho \e(\alpha))
&=
2X\frac{\log\log X}{(1-2\pi i \alpha X)\log X}\sum_{j\leq\sqrt{X}}\frac{1}{j^2}+O\bigg(\frac{X \log \log \log\log X}{\log X}\bigg).	
\end{align*}
Inserting  $\sum_{j\leq\sqrt{X}}\frac{1}{j^2} =\zeta(2) + O(X^{-1/2})$ completes the proof.
\end{proof}	
The previous results from this section now allow us to conclude that the non-principal major arcs will not contribute to the main term and therefore will be absorbed in the error term.
\begin{lemma} \label{culminationnonprincipal}
Let $\mathfrak{M}(q,a)$ be as in \eqref{eq:def_M(q,a)} with $2\leq q\leq (\log X)^A$ and $(a,q)=1$.
We then have for all $\alpha\in\mathfrak{M}(q,a)$ that
\begin{align*}
|\real(\Phi_{\mathbb{P}_2}(\rho \e(\alpha)))| 
\leq  
\frac{3}{4}\Phi_{\mathbb{P}_2}(\rho).
\end{align*}	
\end{lemma}
\begin{proof}
We know from Theorem~\ref{maintermstheorem} that 
\begin{align*} 
\Phi_{\mathbb{P}_2}(\rho) 
&= 
2\zeta(2)\frac{ X}{\log X}(M + \log \log X ) \bigg(1+O\bigg(\frac{1}{\log X}\bigg)\bigg) .
\end{align*}
Since $\alpha\in\mathfrak{M}(q,a)$ with $q\geq 2$, Lemmas~\ref{lemma5.1} and \ref{lem:Vaugah_constanot_major} immediately imply Lemma~\ref{culminationnonprincipal}.
\end{proof}
%%%%%%%%%%%%%%%%%%%%%%%%%%%%%%%%%%%%%%%%%%%%%%%%%%%%%%%%%%%%%%%%%%%%%%
\section{Proof of Theorems \ref{maintheorempapernotrepeated}, \ref{theoremO2}, and \ref{maintheorempaper}} \label{sec:mainresult}
%%%%%%%%%%%%%%%%%%%%%%%%%%%%%%%%%%%%%%%%%%%%%%%%%%%%%%%%%%%%%%%%%%%%%%
We are now in a position to conclude the proof of the main theorems of Section \ref{sec:introduction} by showing that the asymptotict formula for $\partition_{\mathbb{P}_2}(n)$ comes exclusively from the principal major arc $\mathfrak{M}(1,0)$.
\begin{theorem} \label{theorem1over5}
Let $\partition_{\lambda,\mu}(n)$,  $\Psi_{\lambda,\mu}(z)$ and $\Phi_{\lambda,\mu}(z)$ be as in \eqref{eq:def_Psi_lambda_mu} and $\rho$ be the solution 
of the equation 
\begin{align*}
n= \rho\Phi'_{\lambda,\mu} (\rho).
\end{align*}
We then have 
\begin{align}
\partition_{\lambda,\mu}(n)
= 
\frac{\rho^{-n} \Psi_{\lambda,\mu}(\rho)}{(2\pi\Phi_{\lambda,\mu,(2)} (\rho))^{\frac{1}{2}}}(1+O(n^{-\frac{1}{6}}))
\end{align}
as $n\to\infty$ with $\Phi_{\lambda,\mu,(2)} (\rho)= (\rho \frac{d}{d\rho})^2\Phi_{\lambda,\mu}(\rho)$.
\end{theorem}
\begin{proof}
Recall that $\partition_{\lambda,\mu}(n)$ is given by the integral
\begin{align*}
\partition_{\lambda,\mu}(n)
=
\rho^{-n}\int_{-1/2}^{1/2}\Psi_{\lambda,\mu}(\rho \e(\alpha))\e(-n\alpha) d\alpha 
= 
\rho^{-n}\int_{-1/2}^{1/2}\exp(\Phi_{\lambda,\mu}(\rho \e(\alpha))\e(-n\alpha)) d\alpha.    
\end{align*}
We have shown in \eqref{maintermtheorem2} that for $m\in\N_0$ that
\begin{align} 
\left(\rho\frac{\partial}{\partial \rho}\right)^m\Phi_{\mathbb{P}_2}(\rho) 
= 
2\frac{\zeta(2) \Gamma(m+1) X^{m+1}}{\log X}(M + \log \log X ) \bigg(1+O\bigg(\frac{1}{\log X}\bigg)\bigg).
\label{eq:recall_main_alpha=0}
\end{align}
Furthermore, it was shown in \cite[Lemma~3.1]{gafniprimepowers} for $m\in\N_0$ that
\begin{align}
\left(\rho\frac{\partial}{\partial \rho}\right)^m \Phi_{\mathbb{P}^2}(\rho)
=
\frac{\Gamma(m+1/2)\zeta(3/2)X^{m+1/2}}{\log X} \left(1+ O\left(\frac{1}{\log X}\right)\right).
\label{eq:saddle_gafni}
\end{align}
Thus, for any $\alpha\notin\mathfrak{M}(1,0)$ and $n$ sufficiently large,  Lemma \ref{minorarclemmabound} and Lemma \ref{culminationnonprincipal} imply
\begin{align}
\real\big(\Phi_{\lambda,\mu}(\rho \e(\alpha))\big)
&=  
\lambda 	\real\big(\Phi_{\mathbb{P}_2}(\rho \e(\alpha))\big) + 	\mu \real\big(\Phi_{\mathbb{P}^2}(\rho \e(\alpha))\big)
\leq  
\frac{3}{4}\lambda\Phi_{\mathbb{P}_2}(\rho) + |\mu|  \Phi_{\mathbb{P}^2}(\rho)\\
&\leq  
\frac{5}{6}\lambda\Phi_{\mathbb{P}_2}(\rho) + \mu \frac{5}{6} \Phi_{\mathbb{P}^2}(\rho)
=
\frac{5}{6} \Phi_{\lambda,\mu}(\rho).
\label{eq:bound_minor_arc_Phi_lambda_mu}
\end{align}
Further, \eqref{magnitudes3} implies with $x=n$ and $m=0$ that  
$\Phi_{\lambda,\mu}(\rho)\sim  2(\frac{n}{\lambda}\zeta(2)\frac{\log\log n}{\log n})^{1/2}$. 
Thus we get
\begin{align*}
\left|\int_{[-\frac{1}{2}, \frac{1}{2}]\setminus \mathfrak{M}(1,0)}
\exp(\Phi_{\lambda,\mu}(\rho \e(\alpha))\e(-n\alpha))d\alpha \right|
&\leq 
\int_{[-\frac{1}{2}, \frac{1}{2}]\setminus \mathfrak{M}(1,0)}
\exp\big(\real(\Phi_{\lambda,\mu}(\rho \e(\alpha)))\big) d\alpha \\
&\leq 
\exp\left(\frac{5}{6}\Phi_{\lambda,\mu}(\rho)\right)
=
\exp\left(-\frac{1}{6}\Phi_{\lambda,\mu}(\rho)\right)
\exp\left(\Phi_{\lambda,\mu}(\rho)\right)\\
&\ll \left(\Phi_{\lambda,\mu}(\rho)\right)^{-C}\exp\left(\Phi_{\lambda,\mu}(\rho)\right) 
\ll
n^{-B}\exp\left(\Phi_{\lambda,\mu}(\rho)\right),
\end{align*}
where $B\geq 1$ can be chosen arbitrarily large.
Therefore
\begin{align}\label{principalarc}
\partition_{\lambda,\mu}(n)
= 
\rho^{-n}\int_{\mathfrak{M}(1,0)}\exp(\Phi_{\lambda,\mu}(\rho \e(\alpha))\e(-n\alpha)) d\alpha +O(\rho^{-n}\Psi_{\lambda,\mu}(\rho)n^{-B})  
\end{align}
for any constant $B\geq 1$. 
Recall, $\mathfrak{M}(1,0) = \{\alpha\in[-1/2,1/2]; |\alpha|\leq X^{-1}(\log X)^{A}\}$ with $A>18$.
We now spilt $\mathfrak{M}(1,0)$ into the three regions
\begin{align*}
I_1 =\{\alpha\in\mathfrak{M}(1,0);\,|\alpha|\leq \eta \},\ 
I_2 =\{\alpha\in\mathfrak{M}(1,0);\,\eta <|\alpha|\leq \beta\} \ \text{ and } \
I_3 =\{\alpha\in\mathfrak{M}(1,0);\,\beta<|\alpha| \}
\end{align*}
with $\eta = X^{-17/12}$ and $\beta = (9\pi X)^{-1}$.
Next we show that the integrals over $I_2$ and $I_3$ are of lower order.
For $I_3$ we use Corollary~\ref{cor:lemma5.1} and for $I_2$ the Taylor approximation of $\Phi_{\mathbb{P}_2}(\rho \e(\alpha))$ at $\alpha=0$.
We begin with $I_3$. Using that $(1-r)^{-1}\leq 1- 2r$ for $0\leq r\leq 1/2$, we obtain for $\alpha\in I_3$ that
\begin{align*}
\real\left(\frac{1}{1+ 2\pi i \alpha X}\right)
=
\frac{1}{1+ 4\pi^2 \alpha^2 X^2}
\leq 
\frac{1}{1+ 4\pi^2 \beta^2 X^2}
\leq 
1- 2\pi^2 \beta^2 X^2
=
\frac{79}{81}.
\end{align*}
Combining this computation with Corollary~\ref{cor:lemma5.1} and Theorem~\ref{maintermstheorem}, we obtain
\begin{align*}
\real\big(\Phi_{\mathbb{P}_2}(\rho \e(\alpha))\big) 
&\leq 
\frac{79}{81}\left(2X\zeta(2)\frac{\log\log X}{\log X}\right)+O\bigg(\frac{X \log \log \log\log X}{\log X}\bigg)
\leq 
\frac{80}{81} \Phi_{\mathbb{P}_2}(\rho)
\end{align*}
for $X$ large enough. 
To show this, we used $\log \log \log\log X = o(\log\log X)$ and thus for $X$ large
$$\frac{X \log \log \log\log X}{\log X}\leq \frac{1}{81}\frac{\log\log X}{\log X}.$$
Combining this with the same argument as in \eqref{eq:bound_minor_arc_Phi_lambda_mu}, we deduce immediately that for $\alpha\in I_3$ 
\begin{align*}
\real\big(\Phi_{\lambda,\mu}(\rho \e(\alpha))\big)
&\leq  
\frac{161}{162} \Phi_{\lambda,\mu}(\rho).
\end{align*}
We can now use exactly same argument as above to show that the integral over $I_3$ is $\ll \rho^{-n}\Psi_{\mathbb{P}_2}(\rho)n^{-B}$ with $B\geq 1$ arbitrary and thus of lower order.
As next we look at the integral over $I_2$. 
We use here the Taylor approximation of $\Phi_{\lambda,\mu}(\rho e(\alpha))$ since it is more precise in the central region than Corollary~\ref{cor:lemma5.1}. 
We have
\begin{align}
\Phi_{\lambda,\mu}(\rho \e(\alpha)) 
&= 
\Phi_{\lambda,\mu}(\rho)
+2\pi i\alpha a_n
-2\pi^2\alpha^2 b_n +4\pi^3 R(\rho,\alpha)
\label{eq:taylortheorem2}
\end{align} 
with 
$a_n:= \rho\Phi_{\lambda,\mu}'(\rho)$,   $b_n:= \rho\Phi_{\lambda,\mu}'(\rho)+\rho^2\Phi_{\lambda,\mu}''(\rho)$ and
\begin{align*}
\max\{|\real R(\rho,\alpha)|, |\imag R(\rho,\alpha)|\}
\leq 
|\alpha|^3 c_n,
%	=
%	O\left(\alpha^3 n^2 \right).
\end{align*} 
where $c_n:=\rho \Phi_{\lambda,\mu}'(\rho)+ \rho^2 \Phi_{\lambda,\mu}''(\rho)+\rho^3 \Phi'''_{\lambda,\mu}(\rho)$. 
Equations \eqref{eq:recall_main_alpha=0} and \eqref{eq:saddle_gafni} imply that 
\begin{align*}
b_n\sim 4\frac{\zeta(2) X^{3}\log \log X }{\log X}
\quad \text{and} \quad
c_n\sim 12\frac{\zeta(2) X^{4}\log \log X }{\log X}.
\end{align*}
Since $\alpha \in I_2$ (and  thus $ X^{-17/12}\leq |\alpha|\leq (9\pi X)^{-1}$), we get for $X$ large enough that
\begin{align*}
\real (\Phi_{\lambda,\mu}(\rho \e(\alpha))) 
&\leq
\Phi_{\lambda,\mu}(\rho) 
-2\pi^2\alpha^2 b_n +4\pi^3 |\alpha|^3 c_n\\
&\leq 
\Phi_{\lambda,\mu}(\rho) 
-
2\pi^2\alpha^2\frac{\zeta(2) X^{3}\log \log X }{\log X}\left( 3- 26\pi |\alpha|X   \right)\\
&\leq 
\Phi_{\lambda,\mu}(\rho) 
- \frac{2\pi^2\alpha^2	}{27}\frac{\zeta(2) X^{3}\log \log X }{\log X}
\leq 
\Phi_{\lambda,\mu}(\rho) 
- \frac{2\pi^2 X^{1/6}}{27}\frac{\zeta(2) \log \log X }{\log X}.
\end{align*}
Equation \eqref{eq:magnitudes0} implies that $X^{1/6}  \asymp (\frac{n \log n}{\log\log n})^{1/12}$.
This implies that for $n$ large
\begin{align*}
\left|\int_{I_2}
\exp(\Phi_{\lambda,\mu}(\rho \e(\alpha))\e(-n\alpha))d\alpha \right|
\leq 
\exp (\Phi_{\lambda,\mu}(\rho)  - n^{-1/13})
\ll
n^{-B}\exp (\Phi_{\lambda,\mu}(\rho))
\end{align*}
for any constant $B\geq 1$. 

Thus it remains to compute the integral over $I_1=[-\eta,\eta]$.
This computation follows the standard saddle point method, see for instance \cite[Chapter VIII.3]{FlSe09}.
We thus give only the most relevant details.
Proposition~\ref{propositionmagnitude} implies 
\begin{align*}
\eta=X^{-\frac{17}{12}} \asymp \left(\frac{n \log n}{\log\log n}\right)^{-\frac{17}{12}}, \quad
b_n\sim 2\left(\frac{n}{\lambda}\right)^{\frac{3}{2}} \bigg(\frac{\log n}{4\zeta(2)(\log \log n  )}\bigg)^{\frac{1}{2}}
\ \text{ and } \
c_n \ll n^2\frac{\log n}{\log\log n}.
\end{align*}
Inserting this and \eqref{eq:taylortheorem2} in to the integral over $I_2$ and using that $\rho\Phi_{\lambda,\mu}'(\rho) =n$, we deduce
\begin{align*}
\int_{-\eta}^{\eta}
\exp(\Phi_{\mathbb{P}_2}(\rho \e(\alpha))\e(-n\alpha))d\alpha 
&=
\exp\left(\Phi_{\mathbb{P}_2}(\rho)\right)
\int_{-\eta}^{\eta} \exp\left(-2\pi^2\alpha^2 b_n +O(\alpha^3 c_n )\right)   d\alpha
\\
&=
\frac{	\exp\left(\Phi_{\mathbb{P}_2}(\rho)\right)}{\sqrt{b_n}}
\int_{-\eta\sqrt{b_n}}^{\eta\sqrt{b_n}} \exp(-2\pi^2y^2 +O(n^{-1/6} )) dy.
\end{align*}
Note that $\eta\sqrt{b_n} \geq n^{\frac{1}{25}}$. 
Thus the remaining computational steps are straight forward and we omit them.
\end{proof}

For completeness, it is worth to highlight that one could extend Theorem~\ref{theorem1over5} a little bit and replace the error term by a complete asymptotic expansion. 
However, there is limited added value at this point as the resulting expressions are quite involved. 
We thus do not determine it here.

We can now prove
\begin{theorem} \label{thm:section7.2}
Let $\lambda>0$.
We then have as $n\to\infty$
\begin{align}
\partition_{\lambda,\mu}(n) 
%	&= 
%	\frac{\rho^{-n} \Psi_{\mathbb{P}_2}(\rho)}{(2 \pi \Phi_{\lambda,\mu,(2)}(\rho))^{\frac{1}{2}}}(1+O(n^{-\frac{1}{6}})) \nonumber \\
&\sim \mathfrak{c}_1 n^{-\frac{3}{4}} (\log(n/\lambda))^{-\frac{1}{4}} (\log \log n + \mathfrak{c}_2)^{\frac{1}{4}} \nonumber \\
& \quad \times \exp \bigg\{\mathfrak{c}_3 \bigg( \frac{n}{\log(n/\lambda)}\bigg)^{\frac{1}{2}} (\log \log n + \mathfrak{c}_2)^{\frac{1}{2}} \bigg(1+ O\bigg(\frac{\log \log n}{\log n}\bigg)\bigg)\bigg\}, \nonumber
\end{align}
where the constants are given by
\begin{align}
\mathfrak{c}_1 = \frac{(4 \zeta(2))^{\frac{1}{4}}}{2 \sqrt{\pi}\lambda^{\frac{3}{4}}}, \quad \mathfrak{c}_2 = M - \log 2, 
\quad \textnormal{and} \quad \mathfrak{c}_3 = \frac{\lambda+\lambda^{-1}}{\lambda^{\frac{1}{2}}}(4 \zeta(2))^{\frac{1}{2}}.
\end{align}
\end{theorem}
\begin{proof}
The numerator and denominator can be computed using Proposition \ref{propositionmagnitude} and Theorem \ref{theorem1over5}% as
\begin{align*}
\rho^{-n} \Psi_{\mathbb{P}_{\lambda,\mu}}(\rho) &= \exp\bigg(n \log \frac{1}{\rho(n)} + \Phi_{\lambda,\mu}(\rho(n)) \bigg) \nonumber \\
&= 
\exp\bigg\{ (\lambda+\lambda^{-1}) \bigg( \frac{4 \zeta(2) \frac{n}{\lambda} (\log \log(n/\lambda) + M - \log 2   )}{\log (n/\lambda)} \bigg)^{\frac{1}{2}} \bigg(1+ O\bigg(\frac{\log \log n}{\log n}\bigg)\bigg) \bigg\},
\end{align*}
as well as
\begin{align*}
(\Phi_{{\lambda,\mu},(2)}(\rho(n)))^{\frac{1}{2}} 
= 
\sqrt{2} \left(\frac{n}{\lambda}\right)^{\frac{3}{4}} \bigg(\frac{\log(n/\lambda)}{4 \zeta(2) (M - \log 2 + \log \log(n/\lambda) )}\bigg)^{\frac{1}{4}}\bigg(1+ O\bigg(\frac{\log \log n}{\log n}\bigg)\bigg).
\end{align*}
Therefore, we now arrive at the main asymptotic
\begin{align*}
\partition_{\lambda,\mu}(n) 
&= 
\frac{\rho^{-n} \Psi_{\mathbb{P}_2}(\rho)}{(2 \pi \Phi_{\lambda,\mu,(2)}(\rho))^{\frac{1}{2}}}(1+O(n^{-\frac{1}{6}})) \nonumber \\
&\sim \mathfrak{c}_1 n^{-\frac{3}{4}} (\log(n/\lambda))^{-\frac{1}{4}} (\log \log n + \mathfrak{c}_2)^{\frac{1}{4}} \nonumber \\
&\quad \times \exp \bigg\{\mathfrak{c}_3 \bigg( \frac{n}{\log(n/\lambda)}\bigg)^{\frac{1}{2}} (\log \log n + \mathfrak{c}_2)^{\frac{1}{2}} \bigg(1+ O\bigg(\frac{\log \log n}{\log n}\bigg)\bigg)\bigg\}, \nonumber
\end{align*}
where the constants are now identified to be
\begin{align*}
\mathfrak{c}_1 = \frac{(4 \zeta(2))^{\frac{1}{4}}}{2 \sqrt{\pi}\lambda^{\frac{3}{4}}}, \quad \mathfrak{c}_2 = M - \log 2, 
\quad \textnormal{and} \quad \mathfrak{c}_3 = \frac{\lambda+\lambda^{-1}}{\lambda^{\frac{1}{2}}}(4 \zeta(2))^{\frac{1}{2}},
\end{align*}
and this ends the proof of the main asymptotic.
\end{proof}
Employing Theorems \ref{theorem1over5} and \ref{thm:section7.2} the proofs of Theorems \ref{maintheorempapernotrepeated}, \ref{theoremO2}, and \ref{maintheorempaper} now follow by a direct calculation.
%%%%%%%%%%%%%%%%%%%%%%%%%%%%%%%%%%%%%%%%%%%%%%%%%%%%%%%%%%%%%%%%%%%%%%
\section{Proof of Theorem \ref{differencetheorempaper}} \label{sec:mainresult2}
%%%%%%%%%%%%%%%%%%%%%%%%%%%%%%%%%%%%%%%%%%%%%%%%%%%%%%%%%%%%%%%%%%%%%%
The proof will be a consequence of the following result.
\begin{theorem} \label{difference1over5}
Using the notation defined above with $\rho = \rho(n)$, one has
\begin{align*}
\partition_{\lambda,\mu}(n+1)-\partition_{\lambda,\mu}(n) 
=
\frac{\rho^{-n} \log(\frac{1}{\rho}) \Psi_{\lambda,\mu}(\rho)}{(2\pi\Phi_{\lambda,\mu,(2)}(\rho))^{\frac{1}{2}}}(1+O(n^{-1/6})).   
\end{align*}
\end{theorem}
\begin{proof}
Recall that $\rho=\rho(n)$ and let $X$ satisfy $\rho=e^{-1/X}$. We can use \eqref{integralarcs} to write
\begin{align*}
\partition_{\lambda,\mu}(n+1)-\partition_{\lambda,\mu}(n)
=
\int_{-1/2}^{1/2} \rho^{-n}\exp\big(\Phi_{\lambda,\mu}(\rho \e(\alpha)\big)\e(-n\alpha)(\rho^{-1} \e(-\alpha)-1) d\alpha.    
\end{align*}
We remark the following bound $|\rho^{-1}e^{-2\pi i\alpha}-1|\leq e^{1/X}+1\leq 4$. 
The contribution from $|\alpha|>\eta=X^{-2}(\log X)^2$ is
\begin{align*}
\ll \frac{\rho^{-n}\Psi_{\lambda,\mu}(\rho)}{(2\pi\Phi_{\lambda,\mu ,(2)}(\rho))^{\frac{1}{2}}}n^{-B} 
\end{align*}
for any positive constant $B$ by retracing our steps in the proof of Theorem~\ref{theorem1over5}. 
However, on the complementary interval when $|\alpha|\leq \eta$, we obtain
\begin{align*}
\rho^{-1} \e(-\alpha)-1=\exp\left(\frac{1}{X}-2\pi i \alpha\right)-1=\frac{1}{X}+O(\eta)=\frac{1}{X}+O(X^{-2}(\log X)^2).    
\end{align*}
Once again, retracing to the proof of Theorem \ref{theorem1over5}, the segment $[-\eta, \eta]$ yields
\begin{align*}
\int_{-\eta}^{\eta} \rho^{-n}\exp(\Phi_{\lambda,\mu}(\rho \e(\alpha)-2\pi in\alpha))d\alpha 
= 
\frac{\rho^{-n}\Psi_{\lambda,\mu}(\rho)}{(2\pi\Phi_{\lambda,\mu ,(2)}(\rho))^\frac{1}{2}}(1+O(n^{-1/6})).
\end{align*}
Finally we employ \eqref{asympaux02} so that the difference at $n+1$ and $n$ is
\begin{align*}
\partition_{\lambda,\mu}(n+1)-\partition_{\lambda,\mu}(n) 
&= 
\frac{\rho^{-n}\Psi_{\lambda,\mu}(\rho)}{(2\pi\Phi_{\lambda,\mu ,(2)}(\rho))^{\frac{1}{2}}}(1+O(n^{-1/6}))\left(\frac{1}{X}+O(X^{-2}(\log X)^2)\right) \nonumber \\
&=
\frac{\rho^{-n}\log(\frac{1}{\rho})\Psi_{\lambda,\mu}(\rho)}{(2\pi\Phi_{\lambda,\mu,(2)}(\rho))^{\frac{1}{2}}}(1+O(n^{-1/6})),   
\end{align*}
which is what we aimed to show. 
\end{proof}
From Theorem \ref{difference1over5} and Proposition \ref{propositionmagnitude} we immediately obtain Theorem \ref{differencetheorempaper} by setting the parameters $\mu = \lambda = \frac{1}{2}$.
%%%%%%%%%%%%%%%%%%%%%%%%%%%%%%%%%%%%%%%%%%%%%%%%%%%%%%%
\section{Conclusion and future work} \label{sec:futurework}
%%%%%%%%%%%%%%%%%%%%%%%%%%%%%%%%%%%%%%%%%%%%%%%%%%%%%%%
So far we have considered only two primes in different setups. A natural question to ask is how to go beyond semiprimes into powerful almost-primes or smooth numbers. \\

Let $m$ be a positive integer with canonical decomposition $m = p_1^{m_1}p_2^{m_2} \cdots p_r^{m_r}$ where $p_k \in \mathbb{P}$ for $k=1, \cdots, r$ and where $m_1, m_2, \cdots, m_r$ are positive integers. Now let $\Mset = \{p_1^{m_1}p_2^{m_2} \cdots p_r^{m_r} : p_k \in \Pri \textnormal{ for } k=1,\cdots,r\}$ denote the set of $m$-powerful $r$-almost primes. In this situation, $\partition_\Mset(n)$ denotes the partitions of an integer $n$ in terms of a canonical decomposition of a desired integer $m$. If $m=p^k$ with $p \in \Pri$ and $k \in \N$, then this corresponds to the case studied by Gafni \cite{gafniprimepowers}; and if $k=1$, then this corresponds to the case studied by Vaughan \cite{vaughan}. Moreover, if $r=2$ and $m_1=m_2=1$, then this corresponds to the work presented in this paper. Within this framework we have a substantially richer arithmetical structure and we may study questions such as the partitions into integers that are products of, say, three primes or squares of two primes or the product of a prime and the cube of another prime. \\

Another direction is to manufacture prime zeta functions such as almost-prime zeta functions. For instance, we could define the prime zeta function 
\begin{align*}
\zeta_{\mathbb{P}}(k,s) := \sum_{\substack{n \ge 1 \\ \Omega(n)=k}} \frac{1}{n^s},
\end{align*}
where $\Omega(n)$ is the total number of prime factors of $n$. Clearly when $k=1$, this reduces to the prime zeta function, i.e. $\zeta_{\mathbb{P}}(1,s) =\zeta_{\mathcal{P}}(s)$. However, using Newton's identities one observes that
\begin{align*}
\zeta_{\mathbb{P}}(2,s) &= \frac{1}{2}(\zeta_{\mathcal{P}}(s)^2+\zeta_{\mathcal{P}}(2s)) \nonumber \\
\zeta_{\mathbb{P}}(3,s) &= \frac{1}{6}(\zeta_{\mathcal{P}}(s)^3+\zeta_{\mathcal{P}}(s)\zeta_{\mathcal{P}}(2s)+2\zeta_{\mathcal{P}}(3s)) \nonumber \\
\zeta_{\mathbb{P}}(4,s) &= \frac{1}{24}(\zeta_{\mathcal{P}}(s)^4+6\zeta_{\mathcal{P}}(s)^2\zeta_{\mathcal{P}}(2s)+3\zeta_{\mathcal{P}}(2s)^2+8\zeta_{\mathcal{P}}(s)\zeta_{\mathcal{P}}(3s)+6\zeta_{\mathcal{P}}(4s)),
\end{align*}
and in general for $n \in \N$ one has
\begin{align*}
\zeta_{\mathbb{P}}(n,s) = \sum_{\substack{k_1+2k_2+\cdots+nk_n = n \\ k_1, k_2, \cdots, k_n \ge 0}} \prod_{i=1}^n \frac{\zeta_{\mathcal{P}}(is)^{k_i}}{k_i! i^{k_i}},
\end{align*}
effectively making products of powers of $\zeta_{\mathcal{P}}(ks)$ the main component behind the arithmetic of the partitions associated with $\zeta_{\mathbb{P}}(n,s)$. Therefore if we choose $\zeta_{\mathcal{A},w}(s) = \zeta_{\mathbb{P}}(n,s)$, then the analysis of the partitions is essentially reduced to choosing $\zeta_{\mathcal{A},w}(s)=\prod_{k=1}^n \zeta_{\mathcal{P}}(ks)^j$, again with the choice $w(a)=1$ for all $a$.
%%%%%%%%%%%%%%%%%%%%%%%%%%%%%%%%%%%%%%%%%%%%%%%%%%%%%%%%%%%%%%%%%%%

\end{document}